\documentclass[pdflatex,sn-mathphys-ay]{sn-jnl} 

\usepackage{graphicx}%
\usepackage{multirow}%
\usepackage{amsmath,amssymb,amsfonts}%
\usepackage{amsthm}%
\usepackage{mathrsfs}%
\usepackage[title]{appendix}%
\usepackage{xcolor}%
\usepackage{textcomp}%
\usepackage{manyfoot}%
\usepackage{booktabs}%
\usepackage{algorithm}%
\usepackage{algorithmicx}%
\usepackage{algpseudocode}%
\usepackage{listings}%

\theoremstyle{thmstyleone}%
\newtheorem{theorem}{Theorem}%
\newtheorem{corollary}{Corollary}
\newtheorem{proposition}[theorem]{Proposition}%

\theoremstyle{thmstyletwo}%
\newtheorem{remark}{Remark}%

\theoremstyle{thmstylethree}%
\newtheorem{definition}{Definition}%
\newtheorem{lemma}{Lemma}%
\setcitestyle{numbers,square,sort&compress} 

\begin{document}

\title[A successive difference-of-convex method for a class of  two-stage nonconvex  nonsmooth  stochastic conic program via SVI]{A successive difference-of-convex method for a class of  two-stage nonconvex  nonsmooth  stochastic conic program via SVI}


\author*[1,2]{\fnm{Chao} \sur{Zhang}}\email{chzhang2@bjtu.edu.cn}

\author[1]{\fnm{Di} \sur{Wang}}\email{wangdi@bjtu.edu.cn}
\affil[1]{\orgdiv{School of Mathematics and Statistics}, \orgname{Beijing Jiaotong University}, \city{Beijing}, \postcode{100044}, \country{P.R. China}}

\affil[2]{\orgdiv{Beijing Key Laboratory of Biological Big Data and Topological Statistics}, \orgname{Beijing Jiaotong University}, \city{Beijing}, \postcode{100044}, \country{P.R. China}}









\abstract{	We consider a class of two-stage   nonconvex nonsmooth  stochastic conic program,  where the objective functions in both stages can contain nonsmooth terms
		that are 
		functions with easily computed proximal mappings, further composed with affine mappings. This kind of problem is capable of modeling various applications. Solving these problems, however, can be challenging due to the two-stage structure with possibly large number of scenarios,  the nonconvex, nonsmooth and even non-Lipschitz discontinuous terms, as well as  the  conic constraints. In this paper, we define a KKT point of the problem, show that it is a necessary optimality condition under mild conditions, and transform it to an equivalent nonmonotone nonsmooth two-stage stochastic variational inequality (SVI).   We then propose a successive difference-of-convex (SDC) method by making use of Moreau envelope to solve it, the subproblems of which are  approximately solved by the progressive hedging method for solving  maximal monotone two-stage SVI.
		We show the rigorous convergence of our method under suitable assumptions. 
		An extension of Markowitz's mean-variance model is provided
		as an application and numerical results on it demonstrate the effectiveness of the model and the SDC method.}

\keywords{Two-stage nonconvex  nonsmooth stochastic conic program, Stochastic variational inequality, Difference-of-convex method, Progressive hedging method  }


\pacs[MSC Classification]{90C15,90C33,49J40}

\maketitle

\section{Introduction}\label{sec1}

	Given the probability space $(\Omega, {\cal F},P)$ with the finite support set $\Xi_K:=\{\xi_1,\ldots,\xi_K\}$
		for any $\xi_i\in\Xi_K$, let us denote the probability  $p(\xi_i)= p_i$ for each $i\in [K]:=\{1,2,\ldots,K\}$.
Let $r_1(\cdot): \mathbb{R}^{m_{1}}\rightarrow \mathbb{R}$, $r_{\xi_i}(\cdot) : \mathbb{R}^{m_2}\to \mathbb{R}$, $i\in [K]$
		be  closed, lower semicontinuous (lsc), lower bounded functions defined by      
		\begin{eqnarray}\label{r-r} r_1(x) =  P_1(U_1x + u_1),\quad r_{\xi_i}(y_i) =  P_{\xi_i}(U_{\xi_i}y_i+u_{\xi_i}),~ i\in [K],
		\end{eqnarray} 
		where  
		$U_1$, $U_{\xi_i}$, $i\in [K]$ are given matrices, $u_1$, $u_{\xi_i}$, $i\in [K]$ are given vectors, and $P_1$, $P_{\xi_i}$, $i\in [K]$ are given mappings of appropriate dimensions.

	We consider the two-stage nonconvex  nonsmooth stochastic conic program (T-NNS-SCP)  in the following form.  
	\begin{subequations}	
		\begin{equation}\label{stage1}		
			\begin{aligned}
				\min_{x\in {\mathbb{R}}^{m_1}} \ \ &
				c(x)+r_1(x)+ \sum_{k=1}^K p_i[\vartheta(x,\xi_i)]\\
				\mathrm{s.t.} \ \ ~&Ax= a,\\
				&B x - b\in -C_{1},
			\end{aligned}
		\end{equation} 
		where  
		$\vartheta(x,\xi_i)$ is the optimal value of the second-stage problem
		\begin{equation}\label{stage2}
			\begin{aligned}
				\min_{y_i\in \mathbb{R}^{m_{2}}}\ \ &q_{{\xi_i}}(y_i) + r_{\xi_i}(y_i)\\
				\mathrm{s.t.}~~ \ \      & A_{1,\xi_i}y_i+A_{2,\xi_i}x=d_{\xi_i}, \\
				&W_{\xi_i} y_i+T_{\xi_i}x - h_{\xi_i} \in -C_{2,\xi_i}.
			\end{aligned}
		\end{equation}
	\end{subequations}
Here, $A\in\mathbb{R}^{n_{1}\times m_1}$, $B \in \mathbb{R}^{s_1 \times m_1}$, $A_{1,\xi_i}\in \mathbb{R}^{n_{2}\times m_2}$,  $A_{2,\xi_i}\in  \mathbb{R}^{n_{2}\times m_1}$, $W_{\xi_i} \in \mathbb{R}^{s_2 \times m_2}$, and $T_{\xi_i} \in \mathbb{R}^{s_2 \times m_1}$, $i \in [K]$ are given matrices; 
	$a \in \mathbb{R}^{n_{1}}$, $b\in \mathbb{R}^{s_1}$, 
	$d_{\xi_i} \in \mathbb{R}^{n_{2}}$, and $h_{\xi_i} \in \mathbb{R}^{s_2}$, $i\in [K]$ are given vectors; and  $C_{1}\subseteq \mathbb{R}^{s_{1}}$, $C_{2,\xi_i} \subseteq  \mathbb{R}^{s_{2}}$, $i\in [K]$ are symmetric closed convex cones with nonempty interior, any  of which can be a second-order cone, a nonnegative cone, a cone of positive semidefinite matrices, or a Cartesian product of symmetric cones of the  above three types.
	
	Let us denote the feasible set of the first-stage problem (\ref{stage1}) by $X$, the feasible set of the second-stage problem (\ref{stage2}) by $Y(x,\xi_i)$ for each fixed $(x, \xi_i) \in X \times \Xi_K$, and define 
	\begin{eqnarray}\label{Phi}
		\Phi:=\{(x,y_1,\ldots,y_K)\in \mathbb{R}^{m_1 + m_2 K}\ :\ x\in X, ~y_i\in Y(x,\xi_i),~  i\in [K]\}. 
	\end{eqnarray}
	We assume the following assumptions throughout the paper:

\begin{description}
  \item[\bf{A1.}] $P_1$, $P_{\xi_i}$, $i\in [K]$ are closed lsc functions bounded from below, and have easily computed proximal operators.
  
  \item[\bf{A2.}] $c:\mathbb{R}^{m_1}\to\mathbb{R}$ and 
  $q_{\xi_i}:\mathbb{R}^{m_2}\to\mathbb{R}$, $i\in[K]$ are twice continuously differentiable convex functions.
  
  \item[\bf{A3.}] $\Phi$ is a compact set.
  
  \item[\bf{A4.}] For any $(x,\xi_i)\in X\times\Xi_K$, 
  $\sum_{i=1}^K p_i|\vartheta(x,\xi_i)|<+\infty$.
  
  \item[\bf{A5.}] The set $\Phi$  has a Slater point: there exists $(\mathring{x},\mathring{y}_1,\ldots,\mathring{y}_K)$ such that $A \mathring{x} = a$, $A_{1,\xi_i} \mathring{y}_i + A_{2,\xi_i} \mathring{x} = d_{\xi_i}$, $B \mathring{x} -b \in {\rm{int}}({-C_1}), W_{\xi_i}\mathring{y}_i + T_{\xi_i} \mathring{x} - h_{\xi_i} \in {\rm{int}} (-C_{2,\xi_i})$, $i\in [K]$. 	
\end{description}

		The T-NNS-SCP model in (\ref{stage1})-(\ref{stage2}) is promising for modeling important applications under uncertainty. It inherits the merits of two-stage stochastic programming (SP) to deal with uncertainty,  stemming from its ability to make ``here-and-now'' decisions while explicitly accounting for future uncertain events in the first-stage, as well as the ``wait-and-see'' decisions after the uncertainty is resolved in the second-stage.
	Moreover, it can encompass conic constraints, so  it is capable of capturing complex constraints involving nonnegative, second-order, and positive semidefinite cones. 
Most importantly, it allows both the first- and the second-stage objective functions to have nonconvex  nonsmooth regularizers that can be even non-Lipschitz and discontinuous. The nonsmooth terms $P_1$ and $P_{\xi_i}$, $i\in [K]$ have easily computed proximal mappings, and can be further composed with affine maps. Therefore, the composite functions $r_1$ and $r_{\xi_i}$ may not have easily computed proximal terms, which enlarges the range of nonsmooth terms.

Mehrotra and \"Ozevin	
\cite{Sanjay} considered a convex and smooth two-stage stochastic conic programming which is a special case of T-NNS-SCP in (\ref{stage1})-(\ref{stage2}), where $r(x) \equiv  0$, $r_{\xi_i}(y_i) \equiv 0$, $i\in [K]$, $c(\cdot)$ and $q_{\xi_i}(\cdot)$, $i\in [K]$ are convex quadratic functions. A practical primal interior-point decomposition algorithm was developed to solve the problem.  Numerical experiment on an interesting two-stage SP extension of the Markowitz portfolio selection model is given to study the performance of the proposed method. 
	In many applications, nonconvex nonsmooth and even non-Lipschitz discontinuous regularizers are beneficial for inducing desired sparsity or other proper penalties, and the objective functions in the two stages are not quadratic. This motivates us to consider the new T-NNS-SCP model  in this paper.

Although the T-NNS-SCP 
is potential to model important applications, how to solve it efficiently is not an easy task due to the the two-stage structure with possibly large number of scenarios $K$, nonconvex, nonsmooth and even non-Lipschitz discontinuous regularizers in both the first- and the second-stage objective functions, and the complex conic constraints.
The constraints in (\ref{stage2}) are blockwise separable in $y_1, y_2,\ldots, y_K$, and therefore if we write the two-stage problem as one-stage problem of both the first- and the second-stage variables, there is a block-angular structure between the first- and second-stage variables. For such a specially structured problem,  decomposition methods such as the progressive hedging method (PHM)
developed by  Rockafellar and Wets
\cite{RW1991} 
	can be efficient.  However, as pointed out in \cite{Cui}, the convergence of PHM is largely restricted to convex problems and convex-related extensions \cite{zhangLP, ZhangMin}.
    This bottleneck restricts the employment of PHM to general nonconvex two-stage SP problems. 
   The problem of solving  nonconvex two-stage stochastic programming problems has attracted considerable interest in recent years \cite{tssp,Cui,LiuJ,LiuA}.
    Several studies focus on providing an appropriate approximation of the value function at the latest first-stage iterate for each scenario, and then using the approximations to generate the next  first-stage iterate.  
When there are no explicit nonsmooth terms 
in the objectives, and the convex constraints are constructed by linear equations and nonlinear inequalities,  a novel
smoothing method proposed  by Borges, Sagastizábal, and Solodov 
\cite{tssp} can deal with the possibly nonconvex value function of the second-stage, by adding  the Tikhonov-regularized barrier of the inequality constraints of the second-stage.
Recently, Li and Cui 
\cite{Cui}  leveraged the implicit  convex-concave structure of the recourse function, and obtained surrogates of the value function by solving a so-called partial Moreau envelope, which is a  convex subproblem under the  original second-stage constraints at each iterate.

From another point of view, the SP is closely related to the SVI.
The SVI provides a unified mathematical model for optimality conditions of stochastic optimization and stochastic games, and therefore has wide applications in science, engineering, economics, and finance.
\cite{TK} investigated a class of two-stage SVIs, and employed  expected residual minimization to reformulate the two-stage SVI as a two-stage SP with recourse. Rockafellar and Wets first gave  a  definition of multistage SVIs in 
\cite{RoWets}, while Rockafellar and Sun later proposed the PHM for solving  monotone multistage SVIs in    \cite{RoSun}.
The discretization approximation of the two-stage SVI has been investigated by 
\cite{CSS,SunHailin}, 
where the underlying random data are continuously distributed.
Solving the two-stage SP via the SVI approach has been studied in \cite{Lizhang}, where there is no nonsmooth term in the objective function as we consider in this paper, and the two-stage SP in \cite{Lizhang} is convex and its equivalent two-stage SVI is maximal monotone and consequently the PHM can be used to solve it. 
A two-stage nonsmooth stochastic equilibrium model  was proposed for the storage and dynamic distribution of medical supplies in epidemics by \cite{Lizhang2}. By smoothing the nonsmooth convex penalty term for inadequate medical supplies that involves the max operator, the optimistic version of the smoothing problem is equivalently transformed into a two-stage SP, with lots of inequality constraints. The two-stage SP is further transformed equivalently to a smooth two-stage SVI that is of maximal monotone, and the PHM can be employed  to approximately solve it. 

In this paper, we design a novel SDC method to solve the  T-NNS-SCP model 
via the SVI approach, and show that there exists an accumulation point of the SDC method, and any accumulation point provides a KKT point of the original T-NNS-SCP model  under mild assumptions.
To be specific, we first investigate the optimality conditions of the T-NNS-SCP 
and  transform it into an equivalent nonconvex nonsmooth  two-stage SVI. 
Motivated by \cite{Liu}, we use the Moreau envelope of each nonsmooth term 
to obtain an approximate problem, which is easy to obtain through proximal mappings. 
Then in each inner iteration for approximately solving the approximate problem,  we further linearize the second term in the DC reformulation of the Moreau envelope, add regularization terms at the current point, and consequently obtain a smooth and convex two-stage SP.
This problem is equivalent to a maximal monotone two-stage SVI, and is  
approximately solved by the well-known PHM. 
We refer to the combination of SDC and PHM as SDC-PHM for short. 
By employing carefully designed stopping criteria for the approximate solution of the PHM subproblems, an inexact solution strategy for the approximate problem, and a suitable updating rule for the parameter in the Moreau envelope, we establish the convergence of the SDC method under suitable assumptions.
Compared to solving the two-stage SP directly, the merits of solving it via the SVI approach are as follows. Firstly, the method avoids generating implicit nonsmoothness induced by the value function of the second-stage problem. Secondly, it addresses both the primal  and the dual variables (multipliers), so that the SDC method is able to  generate a subsequence that converges to  a KKT point, which is stronger than the limiting  stationary point. 
Thirdly, in each iteration, the SVI corresponding to the approximate problem is solved only roughly, and the iterate points are allowed to be infeasible which  helps reduce computational costs, particularly when dealing with intricate constraints such as linear equations and general conic constraints. The proposed SDC-PHM method can be viewed as an extension of the well-known PHM to a rather general setting of nonconvex and nonsmooth two-stage SP or SVI problems. It is worth noting that the two-stage SVI subproblem in the SDC framework can also be addressed by other efficient two-stage SVI solvers and is not limited to the PHM.

The paper is organized as follows.
In Sect. 2,  we provide notation, some basic definitions and results.
In Sect. 3,  we analyze the T-NNS-SCP, define a KKT point for the T-NNS-SCP that is necessary for optimality under
mild assumptions, and  reformulate the KKT point into an equivalent nonmonotone nonsmooth two-stage SVI.
In Sect. 4,  
we develop the SDC method, and adopt the PHM to solve the subproblems in SDC. We show that SDC-PHM is well-defined and the convergence of SDC is guaranteed 
under suitable assumptions.
In Sect. 5, a nonconvex nonsmooth two-stage extension of Markowitz's mean-variance model is given as an application of the T-NNS-SCP. Numerical experiments are conducted to evaluate both the proposed model and the SDC method.

\section{Notation and preliminaries}

Throughout the paper, we use the following notation. All vectors are assumed to be column vectors. 
The transpose of a matrix \( W \) is denoted by \( W^T \), the $i$-th component of a vector $w$ is denoted by $(w)_i$,
and the Euclidean norm of $w$ is denoted by \( \| w \| \). For $g_i \in \mathbb{R}^{t_i}$, $i\in [K]$, $$(g_i; i\in [K]):=(g_1;g_2;\ldots;g_K) := (g_1^T,g_2^T,\cdots,g_K^T)^T\in \mathbb{R}^{t_1+t_2+\ldots+t_K}.$$  
Given $w \in \Omega$, the tangent  cone  to $\Omega$ at $w$ is denoted as $T_{\Omega}(w)$ \cite[Definition 6.1, pp. 197]{Variational}, and the normal cones in the sense of  regular (also called Fréchet) and Mordukhovich (also called limiting or basic) are denoted as $\hat{N}_{\Omega}(w)$ and $N_{\Omega}(w)$, respectively \cite[Definition 6.3, pp. 199]{Variational}. 
For a cone $\hat{C}$, $	\hat{C}^{o}=\{w : \langle w,w'\rangle\leq 0, {\rm ~for~any~} w'\in \hat{C}\}$ is its
polar cone. 
The domain of a function \( f \) is given by  
$
\operatorname{dom} f := \{w \in \mathbb{R}^{n} : f(w) < +\infty\}$.
The indicator function \( \delta_{\Omega}(w) \) is defined as
$\delta_{\Omega}(w) =0$ if  $w \in \Omega$, and $\delta_{\Omega}(w) = +\infty$ otherwise.
For a differentiable mapping \( L : \mathbb{R}^n \to \mathbb{R}^m \), we denote its
Jacobian matrix at 
\( w \) by \( \mathcal{J}L(w) \in \mathbb{R}^{m \times n}\) and $\nabla L(w):=\mathcal{J}L(w)^T$.   We use $\mathcal{J}_{\textbf{s}}L(w)$ to denote the partial derivative of the mapping $L(w)$ with respect to \textbf{s}, where each component of \textbf{s} comes from the vector $w$.  Given cones $\hat{C}_i$, $i\in [K]$ we denote by $\prod_{i=1}^K \hat{C}_i$ 
the Cartesian product of  cones $\hat{C}_i$, $i\in [K]$, and $(\hat{C})^{\otimes K}:= \underbrace{\hat{C} \times \cdots \times \hat{C}}_{K}$.	

According to 
{\cite[pp. 333,  pp. 366]{Shapiro}}, 
a function $f:\mathbb{R}^n\rightarrow \bar{\mathbb{R}}:=[-\infty,\infty]$ is said to be lower semicontinuous  (lsc) at a point $\hat w\in\mathbb{R}^n$ if $f(\hat w)\leq\liminf_{w\rightarrow \hat{w}}f(w)$. It is said that $f$ is lsc if it is lsc  at every point of $\mathbb{R}^n$. 
By \cite[Definition 8.3, pp. 301]{Variational}, 
for any nonconvex function $f :\mathbb{R}^n\rightarrow \bar{\mathbb{R}}$ and any $\bar{w}\in{\rm dom} f$, the Fréchet, limiting, and horizon subdifferentials  are defined, respectively, as
\begin{equation*}
\begin{aligned}
& \hat{\partial} f(\bar{w}):=\left\{v \in \mathbb{R}^n: {\rm{liminf}}_{w \rightarrow \bar{w},\ w \ne \bar{w}} \frac{f(w)-f(\bar{w})-v^T(w-\bar{w})}{\|w-\bar{w}\|} \geq 0\right\}, \\
& \partial f(\bar{w}):=\left\{v \in \mathbb{R}^n: \exists~ w^k \xrightarrow{f} \bar{w},  ~v^k \rightarrow v~{\rm with}~v^k \in \hat{\partial} f(w^k) {\rm~for~each~}k\right\},\\
&\partial^{\infty} f(\bar{w}):=\left\{v \in \mathbb{R}^n: \exists~ w^k \xrightarrow{f} \bar{w},~ \alpha^k\downarrow 0,~
\alpha^kv^k \rightarrow v~{\rm with}~v^k \in \hat{\partial} f(w^k) {\rm~for~each~}k\right\},	
\end{aligned}
\end{equation*}
where the notation $w^k \xrightarrow{f} \bar{w}$ means that $w^k \rightarrow \bar{w}$ and $f(w^k) \rightarrow f(\bar{w})$. We also define $ \partial f(\bar{w})=\partial^{\infty} f(\bar{w}):=\emptyset$ when $\bar{w}\notin {\rm dom} f$. 
If in addition 
$f(\bar w)$ is finite, then according to \cite[Proposition  8.7, pp. 302]{Variational}, 
$\partial f$ and $\partial^{\infty}f$ are outer semicontinuous, i.e., 
\begin{equation*}\label{subdifferential}
\begin{aligned}
& \left\{v \in \mathbb{R}^n: \exists~ w^k \xrightarrow{f} \bar{w}, ~v^k \rightarrow v~{\rm with}~v^k \in \partial f(w^k) {\rm~for~each~}k \right\}\subseteq\partial f(\bar{w}),\\
&\left\{v \in \mathbb{R}^n: \exists~ w^k \xrightarrow{f} \bar{w},~ \alpha^k\downarrow 0,~
\alpha^kv^k \rightarrow v,~{\rm with}~v^k \in \partial f(w^k) {\rm~for~each~}k\right\} \subseteq\partial^{\infty} f(\bar{w}).
\end{aligned}
\end{equation*}
It is known that $\hat{\partial} f(w) \subseteq \partial f(w)$ (\cite[Theorem 8.6, pp. 302]{Variational}). 
It follows from \cite[Theorem 9.13, pp. 358]{Variational}
that $f$ is Lipschitz
continuous at $\bar{w}$ if and only if 	$\partial^{\infty}f(\bar{w})=\{0\}$.
Moreover, if $f$ is continuously differentiable on a neighborhood of $\bar{w}$, then $\partial f(\bar{w})=\{\nabla f(\bar{w})\}$ according to \cite[Exercise 8.8(b), pp. 304]{Variational}. 

For any real-valued closed lsc function $P_0$ that is bounded below and any $\rho >0$, the Moreau envelope \cite[Definition 1.22, pp. 20]{Variational} is defined as:
\begin{eqnarray*}
P_{0,\rho}(w):=\inf_{w'}\left\{\frac{1}{2\rho}\|w'-w\|^2+P_0(w')\right\}.\nonumber
\end{eqnarray*}
The infimum in the definition of the Moreau envelope is attained at the so-called proximal mapping of $\rho P_0$ at $w$, which is defined as
\begin{equation}\label{proximal}
{\rm prox}_{\rho P_0}(w):=\mathop{\mathrm{arg\,min}}_{w'}\left\{\frac{1}{2\rho}\|w'-w\|^2+P_0(w')\right\}.\nonumber
\end{equation}
This set is always nonempty because $P_0$ is proper closed and bounded below, according to  \cite[Theorem 1.25, pp. 21]{Variational}.
It is easy to see that for any $w$,
\begin{eqnarray}\label{rho-rhot}
-\infty  <  \inf_{w'} P_0(w') \le P_{0,\rho}(w) \le P_0(w).
\end{eqnarray} 
Furthermore, we have the following lemma from \cite[Lemma  1]{Liu}.
\begin{lemma}\label{prox}
Let $P_0$ be a real valued closed function with $\inf P_0 > -\infty$.
Suppose that $w^t \rightarrow w^*, \rho_t \downarrow 0$ and pick any $\zeta^t \in \operatorname{prox}_{\rho_t P_0}\left(w^t\right)$ for each $t$. Then it holds that 
$\zeta^t \rightarrow w^*$.
\end{lemma}

The Moreau envelope can be written as a DC function:
\begin{equation}\label{r_rho}
P_{0,\rho}(w)=\frac{1}{2\rho}\|w\|^2-\underbrace{\max _{w'} 
	\left\{
	\frac{1}{\rho} w^T w'-\frac{1}{2 \rho}\|w'\|^2-P_0(w')
	\right\}}
_{D_{\rho,P_0}(w)},
\end{equation}
where $D_{\rho,P_0}$ is a convex function because it is the maximum of affine functions of $w$. Moreover, using the definition of $P_{0, \rho}(w)$, ${\rm prox}_{\rho P_0}(w)$ and (\ref{r_rho}), we see that the maximum in $D_{\rho, P_0}(w)$ is obtained at any point in ${\rm prox}_{\rho P_0}(w)$.

The lemma  below gives a way to compute a subgradient  of $D_{\rho,P_0}(w)$ as well as an element of limiting subdifferential of $P(w)$ from {\rm\cite[Lemma 1]{Xu}}.
\begin{lemma}\label{subgradient}
$\frac{1}{\rho}{\rm prox}_{\rho P_0}(w)\subseteq\partial D_{\rho,P_0}(w)$, and
$\{\frac{1}{\rho}(w-\zeta):\zeta\in{\rm prox}_{\rho P_0}(w)\}\subseteq {\partial} P_0(\zeta)$.
\end{lemma} 

Let $U_0$ and $u_0$ be arbitrary given matrices and vectors of appropriate size, and let $r_0$ be an arbitrary proper lsc composite function in the form of 
\begin{eqnarray*}
r_0(s)= P_0(U_0s +u_0).
\end{eqnarray*}
Denote
\begin{eqnarray}\label{R0rho}
R_{0,\rho}(s) := D_{\rho, P_0}(U_0s + u_0)\quad \mbox{and}\quad r_{0,\rho}(s) := P_{0,\rho}(U_0s+u_0).
\end{eqnarray}
Using similar arguments to \cite[Section 3.1]{Liu}, we can show that 
\begin{eqnarray}\label{8}
\frac{1}{\rho} U_0^T {\rm prox}_{\rho P_0}(U_0s + u_0) \subseteq U_0^T \partial D_{\rho,P_0}(U_0s+u_0) = \partial R_{0,\rho}(s).
\end{eqnarray}

\section{ Optimality condition for T-NNS-SCP}

In this section, we first describe  the 
T-NNS-SCP in the form of a large-scale conic program  and show the existence of solutions. 

We define a KKT point of the T-NNS-SCP, which is a necessary optimality condition
under  the Robinson constraint qualification (RCQ). We eventually transform it equivalently to a two-stage nonconvex nonsmooth SVI. This motivates us to design our SDC algorithm via SVI in the next section.
To express the RCQ and the SVI in a concise way, 
we define the following notation.
\begin{table}[htbp]
\centering
\begin{tabular}{p{8cm} p{4cm}}
	\toprule
	\textbf{Notation} & \textbf{Description} \\
	\midrule
	$ \mathcal{K} := \{0_{n_1}\} \times -C_1 \times (\{0_{n_2}\})^{\otimes K} \times \prod_{i=1}^K -C_{2,\xi_i}$ & convex cone relating to constraints of both stages \\
	$ D := \mathbb{R}^{m_1}\times (\mathbb{R}^{m_2})^{\otimes K}
	\times \mathcal{K}^{\circ}$
	& convex cone relating to primal and dual variables \\
	\midrule
	$\mathbf{y}: = (y_i;i\in [K])\in (\mathbb{R}^{m_2})^{\otimes K}$ &  vector of second-stage primal variables for all scenarios \\
	${\bf x}_{\bf y}:=(x;\mathbf{y}) \in \mathbb{R}^{m_1}\times (\mathbb{R}^{m_2})^{\otimes K}$ &  vector of primal variables in both stages\\
	$M:= (\alpha_1;\alpha_2;\pi_{1,i};i\in [K];\pi_{2,i};i\in [K]) \in 
	{\cal K}^{\circ}$ &  vector of dual variables in both stages\\
	$\mathbf{z}: = 
	({\bf x}_{\bf y};M)\in D$ &  vector of primal and dual variables in both stages\\
	\midrule
	$\hat{f}(\mathbf{x}_{\mathbf{y}}): = c(x)+\sum_{i=1}^K p_i q_{\xi_i}(y_i)$ & smooth part of objective functions in both stages\\
	$\Psi(\mathbf{x}_{\mathbf{y}}):= r_1(x) + \sum_{i=1}^{K} p_i r_{\xi_i}(y_i)$ & nonsmooth part of objective functions in both stages\\
	${\cal U}({\bf x}_{\bf y}) := (U_1 x+u_1;U_{\xi_i} y_i + u_{\xi_i}; \forall i\in [K])$ & vector-valued function for inner mapping of $\Psi$\\
	$
	P({\cal U}(\textbf{x}_{\textbf{y}})):= P_1(U_1 x + u_1)+ \sum_{i=1}^K p_i P_{\xi_i}(U_{\xi_i} y_i + u_{\xi_i})$ & composite expression of $\Psi({\bf x}_{\bf y})$\\
	\vspace{1mm}
	$G({\bf x}_{\bf y}) := \left(
	\begin{array}{c}
		G^1({\bf x}_{\bf y}) := Ax-a\\
		G^2({\bf x}_{\bf y}) := Bx-b\\
		G^{1,\xi_i} ({\bf x}_{\bf y}):= A_{1,\xi_i} y_i + A_{2,\xi_i} x - d_{\xi_i}; \forall i\in [K]\\
		G^{2,\xi_i}({\bf x}_{\bf y}) := W_{\xi_i} y_i + T_{\xi_i }x -h_{\xi_i}; \forall i\in [K]
	\end{array}
	\right)$ & \vspace{1mm}
	vector-valued function relating to constraints of both stages\\
	\bottomrule
\end{tabular}
\caption{notation description}
\label{tab:symbols}
\end{table}

\begin{proposition}\label{prop3.1}The T-NNS-SCP {\rm(\ref{stage1})-(\ref{stage2})} is equivalent to the following
nonconvex nonsmooth stochastic conic program 
\begin{equation}\label{new_prob}
\begin{aligned}
	\min_{\bf{x_y}} &~{\hat f}({\bf x}_{\bf y})+
  \Psi({\bf x}_{\bf y})
	\\
	\mathrm{s.t.}&~ G(	\bf{x_y})\in \mathcal{K}.
\end{aligned}
\end{equation}
The equivalence is understood in the sense that optimal values of  problem  T-NNS-SCP in {\rm(\ref{stage1})-(\ref{stage2})} and {\rm(\ref{new_prob})} are equal to each other and finite. 
Moreover, the set of optimal solutions of the  T-NNS-SCP in {\rm(\ref{stage1})-(\ref{stage2})}
is nonempty and bounded.
\end{proposition}
\begin{proof} 
We can easily get the optimal values of problem  T-NNS-SCP in (\ref{stage1})-(\ref{stage2}) and (\ref{new_prob}) are equal to each other.
Because the objective function in (\ref{new_prob}) is lower bounded and lsc, and according to assumptions A3 and A4, we can conclude that   the set of optimal solutions  of (\ref{new_prob}) is nonempty and bounded. By 
assumption A4, the optimal value of (\ref{new_prob}) is finite. 
Hence, the common optimal value of problems (\ref{stage1})-(\ref{stage2})  and (\ref{new_prob}) is finite. 
It is clear that if $\bar{\mathbf{x}}_{\bar{\mathbf{y}}}$ is an optimal solution of problem (\ref{new_prob}), then $\bar x$ is an optimal solution of (\ref{stage1}) and $\bar y_i $ is an optimal solution of the second-stage problem (\ref{stage2}) for all $i\in [K]$. 
Therefore, the solution set of T-NNS-SCP in  (\ref{stage1})-(\ref{stage2}) is nonempty. The boundedness of the solution set comes from the compactness of $\Phi = \{{\bf x}_{\bf y}\ :\ G({\bf x}_{\bf y}) \in {\cal K}\}$ in assumption A3.
\end{proof}

Let  ${\bf x}_{\bf y}^* \in \Phi$ be a given point. 
Clearly, $N_{\{0_{n_1 }\}}(G^1(\textbf{x}^*_{\textbf{y}}))=\mathbb{R}^{n_1}$, 
and $ N_{\{0_{n_2}\}}(G^{1,\xi_i}(\textbf{x}^*_{\textbf{y}}))=\mathbb{R}^{n_2}$ for any $i\in[K]$. 
According to \cite[Proposition 1.4, pp. 4]{BorisSMordukhovich}, we have
\begin{equation}\label{KK}
\begin{aligned}
&N_\mathcal{K}(G(\textbf{x}^*_{\textbf{y}}))
= \mathbb{R}^{n_1}\times N_{\{-C_{1}\}}(G^2(\textbf{x}^*_{\textbf{y}}))\times(\mathbb{R}^{n_2})^{\otimes K}\times\prod_{i=1}^K N_{\{-C_{2,{\xi_i}}\}}\left(G^{2,\xi_i}(\textbf{x}^*_{\textbf{y}})\right).\\
\end{aligned}
\end{equation}
By
\cite[(3.9), pp. 147]{Pertur}, 
we have
\begin{equation}\label{complement}
\begin{aligned}
M^*\in N_\mathcal{K}(G(\textbf{x}^*_{\textbf{y}})) &\Longleftrightarrow G({\bf x}_{\bf y}^*)\in N_{{\cal K}^{\circ}} (M^*)\\
&\Longleftrightarrow G(\textbf{x}^*_{\textbf{y}})\in \mathcal{K},~M^*\in \mathcal{K}^\circ, \ {M^*}^TG(\textbf{x}^*_{\textbf{y}}) =0.\\
\end{aligned}
\end{equation}

The RCQ in the following definition is a classic constraint qualification used in conic programming; see e.g.,  
\cite{ZhangLW1}. 
\begin{definition}
We say that the RCQ holds at 
$\textbf{x}^*_{\textbf{y}}\in\Phi$ 
if the gradients $\nabla (G^1(\textbf{x}^*_{\textbf{y}}))_{j_1}$ and $\nabla (G^{1,\xi_i}(\textbf{x}^*_{\textbf{y}}))_{j_2},~i\in[K],~j_1\in[n_1],~j_2\in[n_2]$
are linearly independent  and there exists a vector $d\in\mathbb{R}^{m_1+m_2K}$ such that
\begin{equation}\label{RCQ}
\begin{aligned}
	& \mathcal{J} G^1(\textbf{x}^*_{\textbf{y}})d=0,\\
    &
	\mathcal{J} G^{1,\xi_i}(\textbf{x}^*_{\textbf{y}}) d=0,~i\in[K],\\
&G^2(\textbf{x}^*_{\textbf{y}})+\mathcal{J}G^2(\textbf{x}^*_{\textbf{y}})d\in{\rm int}(-C_1),\\ 
&	G^{2,{\xi_i}}(\textbf{x}^*_{\textbf{y}})+\mathcal{J}G^{2,{\xi_i}}(\textbf{x}^*_{\textbf{y}})d\in {\rm int}(-C_{2,{\xi_i}}),~i\in[K].
\end{aligned}
\end{equation}
\end{definition}	

Assumption A5 means that there is a Slater point for the constraints $G({\bf x}_{\bf y})\in {\cal K}$, which   indeed  guarantees that  the RCQ holds at any feasible point of ${\bf{x}_y} \in \Phi$, according to \cite[Proposition 2.104, pp. 73]{Pertur}.

\begin{lemma}\label{normalcone}
	{\rm
		At any point ${\bf{x_y^*}}\in \Phi$, we have 
		\begin{eqnarray*}
			& &T_{\Phi}(\textbf{x}^*_{\textbf{y}})=\bigg\{d\in\mathbb{R}^{m_1+m_2K}:\mathcal{J}G(\textbf{x}^*_{\textbf{y}})d\in T_\mathcal{K}(G(\textbf{x}^*_{\textbf{y}}))\bigg\}, 
            \\
			& &N_\Phi(\textbf{x}^*_{\textbf{y}})=\hat{N}_\Phi(\textbf{x}^*_{\textbf{y}})=T_{\Phi}^{\circ}(\textbf{x}^*_{\textbf{y}})= \nabla G(\textbf{x}^*_{\textbf{y}})N_\mathcal{K}(G(\textbf{x}^*_{\textbf{y}})).
	\end{eqnarray*}}
\end{lemma}	
\begin{proof}
	The formula for the tangent cone
    comes from {\rm\cite[Corollary 2.91,  pp. 66]{Pertur}} using the fact that the  RCQ holds at any ${\bf{x_y^*}}$ for the system $G({\bf x}_{\bf y})\in {\cal K}$. 
	Furthermore, the first equality 
    for the normal cone is  obtained from \cite[Theorem 6.9, pp. 203]{Variational} since $\Phi$ is convex, the second equality holds by \cite[Theorem 6.28, pp. 220]{Variational}, and the last equality holds according to \cite[(4), pp. 107]{RCQnormalcone}. 
\end{proof}

\begin{definition}[KKT Point]		\label{stationarity} 
	We say that $\textbf{x}^*_{\textbf{y}}$
	 is a  KKT point of problem (\ref{new_prob}) if $\textbf{x}^*_{\textbf{y}}\in \Phi$
	  and there exist vectors $\lambda^*\in \partial P({\cal U}({\bf x}_{\bf y}^*))$
	   and $M^* \in N_{\cal K} (G({\bf x}_{\bf y}^*))$ such that 
	\begin{equation}\label{SPoint}
		0=\nabla \hat{f}(\textbf{x}^*_{\textbf{y}})
		+ \nabla{\cal U}(\textbf{x}_{\textbf{y}}^*) 
		\lambda^* 	+\nabla G(\textbf{x}^*_{\textbf{y}}) M^*.
	\end{equation}	 
\end{definition}

Let ${\bf x}_{\bf y}^* \in \Phi$. To show that a KKT point of (\ref{new_prob}) is a necessary optimality condition for (\ref{new_prob}), we need the following two assumptions. 
\begin{description}
    \item[\bf{A6.}] 
$	-\partial^{\infty} \Psi({\bf x}_{\bf y}^*) \cap N_{\Phi}({\bf x}_{\bf y}^*) = \{0\}.
	$
	
	\item[{\bf A7.}]
$	\nabla {\cal U}({\bf{x}}_{\bf{y}}^*) \lambda =0,\lambda \in \partial^{\infty} P({\cal U}({\bf{x}}_{\bf{y}}^* ))\ \Longrightarrow\   \lambda =0.$
\end{description}

	The condition in assumption \textbf{A6} is called the basic  qualification (BQ), and is widely adopted in nonconvex nonsmooth  optimization; see \cite[Definition 2.3]{RCPLD1}. The condition in assumption {\bf A7}
	is also a classical constraint qualification for nonconvex nonsmooth optimization problems; see \cite[(12)]{Liu}.  If $P_1$ and $P_{\xi_i}$ for all  $i\in [K]$ are all locally Lipschitz continuous, then both $P$ and $\Psi$ are also locally Lipschitz continuous. This guarantees  that $\partial^{\infty} \Psi({\bf x}_{\bf y}^*) = \{0\}$ and $\partial^{\infty}P({\cal U}({\bf x}_{\bf y}^*))=\{0\}$. Hence, in this case,  assumptions {\bf A6} and {\bf A7} obviously hold. 
	 If $U_1$ and $U_{\xi_i}$, $i\in [K]$ are of full row rank, then assumption {\bf A7} holds.

\begin{theorem}\label{necessary optimality 1}
Let {\rm $\textbf{x}^*_{\textbf{y}}$} be a local optimal solution of  problem {\rm(\ref{new_prob})}.
Suppose assumptions {\bf A6} and {\bf A7} hold, 
then {\rm $\textbf{x}^*_{\textbf{y}}$} is a KKT point  of problem {\rm(\ref{new_prob})}.
\end{theorem}
\begin{proof} 
It is easy to see that (\ref{new_prob}) is equivalent to	
\begin{equation*}
\min \hat{f}(	\textbf{x}_{\textbf{y}})+\Psi(	\textbf{x}_{\textbf{y}})+\delta_{\Phi}(	\textbf{x}_{\textbf{y}}). 
\end{equation*}
It follows from Fermat’s rule \cite[Theorem 10.1, pp. 422]{Variational}
that	 
\begin{equation}\label{2-1}
0\in\nabla \hat{f}(\textbf{x}^*_{\textbf{y}})+\partial\left(\Psi(\textbf{x}^*_{\textbf{y}})+\delta_{\Phi}(\textbf{x}^*_{\textbf{y}})\right).
\end{equation}	
Using 
assumption {\bf A6}, $\partial \delta_{\Phi}({\bf x}_{\bf y}^*) = \partial^{\infty} \delta({\bf x}_{\bf y}^*) = N_{\Phi}({\bf x}_{\bf y}^*)$ and the sum rule of the limiting subdifferential \cite[Corollary 10.9, pp. 430]{Variational}, we can get
\begin{equation}\label{2-2}
\begin{aligned}
	\partial\left(\Psi(\textbf{x}^*_{\textbf{y}})+\delta_{\Phi}(\textbf{x}^*_{\textbf{y}})\right)
	\subseteq  	\partial\Psi(\textbf{x}^*_{\textbf{y}})+\partial\delta_{\Phi}(\textbf{x}^*_{\textbf{y}})
	= \partial\Psi(\textbf{x}^*_{\textbf{y}})+N_{\Phi}(\textbf{x}^*_{\textbf{y}}),
\end{aligned}
\end{equation}	 
according to \cite[Exercise 8.14, pp. 310]{Variational}.
Furthermore, by 
assumption {\bf A7} and \cite[Theorem 10.6, pp. 427]{Variational}, we know that 
\begin{eqnarray}\label{Partial_composite}
\partial \Psi({\bf{x}}_{\bf y}^*) \subseteq  \nabla {\cal U}({\bf{x}}_{\bf y}^*)\partial P({\cal U}({\bf x}_{\bf y}^*)).
\end{eqnarray}
Since the RCQ holds
at $\textbf{x}^*_{\textbf{y}}$ for the system $G({\bf x}_{\bf y}) \in {\cal K}$, it follows from Lemma \ref{normalcone} that
\begin{equation}\label{2-3}
N_\Phi(\textbf{x}^*_{\textbf{y}})=\nabla G(\textbf{x}^*_{\textbf{y}})N_\mathcal{K}(G(\textbf{x}^*_{\textbf{y}})).
\end{equation}
Combining (\ref{2-1})-(\ref{2-3}), we find that $\textbf{x}^*_{\textbf{y}} \in \Phi$ satisfies (\ref{SPoint}) and consequently 
$$0\in \nabla \hat f({\bf x}^*) + \nabla {\cal U}({\bf x}_{\bf y}^*) 
\partial P({\cal U}({\bf x}_{\bf y}^*)) + \nabla G({\bf x}_{\bf y}^*) 
N_{\cal K}(G({\bf x}_{\bf y}^*)).$$ 
This shows that there exist $\lambda^* \in \partial P(\mathcal{U}(\mathbf{x}_{\mathbf{y}}^*))$ and $M^* \in N_{\mathcal{K}}(G(\mathbf{x}_{\mathbf{y}}^*))$ satisfying \eqref{SPoint}. Hence $\mathbf{x}_{\mathbf{y}}^*$ is a KKT point.
\end{proof}

For special non-Lipschitz functions $\Psi({\bf x}_{\bf y})$, a local optimal solution ${\bf x}_{\bf y}^* $  of problem (\ref{new_prob}) is a stationary point of (\ref{new_prob}), without assumptions {\bf A6} and {\bf A7}. 
Similarly to \cite[Theorem 2.2]{RCPLD1}, we can show the following corollary without difficulty.
\begin{corollary}\label{separable}
Suppose
$\Psi({\bf x}_{\bf y}) = \sum_{i=1}^{m_1+m_2 K} \psi_i(({\bf x}_{\bf y})_i)$ for some lsc functions $\psi_i : \mathbb{R} \to \mathbb{R}$. Let ${\cal I} = \{i\ : \partial^{\infty} \psi_i(({\bf x}_{\bf y}^*)_i)=\{0\}\}$ and ${\cal I}^c$ be the complement of ${\cal I}$ with respect to $\{1,\ldots,m_1 + m_2 K\}$.
Assume further that for any $i\in {\cal I}^c$, $\partial \psi_i(({\bf x}_{\bf y}^*)_i) = \mathbb{R}$, and the RCQ holds at $({\bf x}_{\bf y}^*)_{\cal I}$ for the system $G(({\bf x}_{\bf y})_{\cal I}, ({\bf x}_{\bf y}^*)_{{\cal I}^c})\in {\cal K}$.
Then ${\bf x}_{\bf y}^*$ is a KKT point of problem {\rm (\ref{new_prob})}.  
\end{corollary}

Clearly, the sparsity-induced penalties $\Psi({\bf x}_{\bf y}): = \|x\|_p^p + \sum_{i=1}^K p_i \| y_i\|_p^p$ for $p\in (0,1)$ and $\Psi({\bf x}_{\bf y}) := \|{x}\|_0 + \sum_{i=1}^K p_i\|y_i\|_0$ satisfy the requirements of $\Psi$ in Corollary \ref{separable}, 
which are non-Lipschitz and the latter one is even discontinuous.

\begin{remark} When the conic constraints 
can be defined by a finite number of equality and inequality constraints,
then the RCQ is equivalent to the Mangasarian-Fromovitz
constraint qualification (MFCQ) according to \cite[pp. 71]{Pertur}.	
Using \cite[Proposition 2.2]{RCPLD1}, the relaxed constant positive linear dependence (RCPLD), which is weaker than the RCQ, is sufficient to derive the  representation of $N_\Phi(\textbf{x}^*_{\textbf{y}})$ 
in Lemma \ref{normalcone}. Consequently, we can derive the same conclusions as in  Theorem \ref{necessary optimality 1} and Corollary \ref{separable} using the RCPLD.
\end{remark}

We further reformulate  the KKT condition in (\ref{SPoint}) to an equivalent non-monotone nonsmooth  two-stage SVI problem, which provides the clue for developing the SDC method with PHM via the SVI approach in the next section.

Letting
\begin{equation*}
	\begin{array}{c}
		H(	\textbf{z};\lambda^*)
	\end{array}:=
	\left
	(
	\begin{array}{c}
		\nabla \hat{f}(\textbf{x}_{\textbf{y}})
			 +	\nabla G({\bf{x}_{\bf y}}) M 
			 + \nabla {\cal U}({\bf x}_{\bf y}) \lambda^* \\
		-G(\textbf{x}_{\textbf{y}})
	\end{array}
	\right
	),~
\end{equation*}
the condition  (\ref{SPoint}) is equivalent to that there exist $\textbf{z}^* = ({\bf x}_{\bf y}^*;M^*)\in D$ and 
$\lambda^* \in \partial P({\cal U}({\bf x}_{\bf y}^*))$ 
such that 
$\textbf{z}^*
$ is a solution of  the following two-stage SVI
\begin{equation}\label{OKKT}
0\in
H({\bf z};\lambda^*)+N_{D}(	\textbf{z}).
\end{equation}
Let us define the natural residual
function by 
\begin{equation}
\begin{aligned}
	&H_{D}^{\rm nat}(\textbf{z};\lambda^*):=\textbf{z}-{\bf Proj}_{D}\left(\textbf{z}-	H(	\textbf{z};\lambda^*)\right),
\end{aligned}	
\end{equation}
where ${\bf Proj}_{D}(\hat{\textbf{z}})$ is the projection of $\hat{\textbf{z}}$ on $D$. Let us denote the set of solutions to (\ref{OKKT}) by SOL$(D,H)$.
By \cite[Proposition 1.5.8, pp. 83]{Finite-Dimensional}, $\textbf{z}^*\in$ SOL$(D,H)$  if and only if
\begin{equation*}
H_{D}^{\rm nat}(\textbf{z}^*;\lambda^*)=0.
\end{equation*}

\section{Solution method for (\ref{new_prob})} \label{section-alg}

\subsection{Successive difference-of-convex method}

We first approximate the nonconvex nonsmooth terms through the Moreau envelope, 
and consider a 
	Moreau envelope  induced approximation problem (MEAP) of  (\ref{new_prob}) defined by
\begin{equation*}\label{DCstage1}
	\begin{aligned}
		\min_{	\textbf{x}_{\textbf{y}}} &~F_{\rho}(	\textbf{x}_{\textbf{y}})
		:={\hat f}({\bf x}_{\bf y}) + r_{1,\rho}(x) + \sum_{i=1}^K p_i r_{\xi_i,\rho}(y_i)
		\\
		\mathrm{s.t.}&~ G(	\textbf{x}_{\textbf{y}})\in \mathcal{K}.
	\end{aligned}\tag{MEAP}
\end{equation*}	
Here,  $r_{1,\rho}$ and $r_{\xi_i,\rho}$, $i\in [K]
$ are the Moreau envelopes of $r_1$ and $r_{\xi_i}$, $i\in [K]$, respectively. They have 
the DC decomposition:
\begin{eqnarray*}
	r_{1,\rho}(x) &=& \frac{1}{2 \rho}\|U_1x +u_1\|^2-R_{1,\rho}(x),\\
	r_{\xi_i,\rho}(y_i)
	&=& \frac{1}{2\rho} \|U_{\xi_i}y_i+u_{\xi_i}\|^2 - R_{\xi_i,\rho}(y_i),\ i\in [K], 
\end{eqnarray*}
where 
$$R_{1,\rho}(x)=D_{\rho,P_1}(U_1 x +u_1) \quad\mbox{and}\quad 
R_{\xi_i,\rho}(y_i)= D_{\rho,P_{\xi_i}}(U_{\xi_i}y_i + u_{\xi_i}).$$ 
The Moreau envelopes $r_{1,\rho}$, $r_{\xi_i,\rho}$, $i\in [K]$ of $r_1$ and $r_{\xi_i}$, $i\in [K]$,  along with their DC decompositions for handling nonconvex nonsmooth 
functions are inspired by  \cite{Liu}. 

We further design a linearized approximation with a regularization term for (\ref{DCstage1}). Given  $\rho_t$,  ${\bf x}_{\bf y}^t = (x^t;y_i^t;
	i\in [K])$, and   
$$\zeta^t_1 \in {\rm prox}_{\rho_t P_1}(U_1 x^t + u_1),\quad \zeta_{\xi_i}^t \in {\rm prox}_{\rho_t P_{\xi_i}} (U_{\xi_i}y_i^t + u_{\xi_i}),\ i\in [K],$$  we know by (\ref{8}) that $\varpi^t := \left(\varpi_1^t;\varpi_{\xi_i}^t; \forall i\in [K]\right)$ satisfies
$$\varpi^t_1 := \frac{1}{\rho_t} U_1^T \zeta_1^t  \in\partial R_{1,\rho_t}(x^t), \quad \varpi_{\xi_i}^t := \frac{1}{\rho_t} U_{\xi_i}^T \zeta_{\xi_i}^t \in \partial R_{\xi_i,\rho_t}(y_i^t),\ i\in [K].$$ 
 For given $\tau:=(\tau_1;\tau_2)>0$, we define the convex smooth  surrogates of $r_{1,\rho_t}(x)$ and $r_{\xi_i,\rho_t}(y_i)$, $i\in [K]$  as 
	\begin{eqnarray*}
		{\hat r}_{1,\rho_t,\tau_1}(x) &:=&\frac{1}{2\rho_t}\|U_1 x+u_1\|^2 -R_{1,\rho_t}(x^t) - \langle \varpi_1^t,x-x^t\rangle + \frac{\tau_1}{2}\|x-x^t\|^2,\\
		{\hat r}_{\xi_i,\rho_t,\tau_2}(y_i) &:=& \frac{1}{2\rho_t}\|U_{\xi_i} y_i + u_{\xi_i}\|^2 - R_{\xi_i,\rho_t}(y_i^t) - \langle \varpi_{\xi_i}^t, y_i - y^t_i\rangle + \frac{\tau_2}{2}\|y_i- y^t_i\|^2, 
	\end{eqnarray*}
	using the linearization of the second convex term in the DC decomposition and a quadratic regularization term at ${\bf x}_{\bf y}^t$, respectively. 
	We then define
	\begin{eqnarray*}
		F^L_{\rho_t,{\bf x}_{\bf y}^t,\varpi^t,\tau}(	\textbf{x}_{\textbf{y}}):=& & \hat f({\bf x}_{\bf y})+{\hat r}_{1,\rho_t,\tau_1}(x) + \sum_{i=1}^K p_i \hat r_{\xi_i,\rho_t,\tau_2}(y_i),
\end{eqnarray*}
and find  that
	\begin{eqnarray}\label{ineq1}
		F_{\rho_t}({\bf{x}}_{\bf{y}}) &\le& F_{\rho_t,{\bf x}_{\bf y}^{t},\varpi^{t},\tau}^L({\bf{x}}_{\bf{y}}) - \frac{\tau_1}{2}\|x-x^{t}\|^2 - \frac{\tau_2}{2}\sum_{i=1}^K p_i\|y_i - y_i^t\|^2,\\
    \label{ineq2}
		F_{\rho_t}(\textbf{x}_{\bf y}^{t}) &=& F_{\rho_t, {\bf x}_{\bf y}^{t},\varpi^{t},\tau}^L(\textbf{x}_{\bf y}^{t}). 
\end{eqnarray}
The linearized approximation problem  of  (\ref{DCstage1}) at ${\bf x}_{\bf y}^t$  can be written as 
\begin{equation*}\label{step2}
	\begin{aligned}
		\min_{	\textbf{x}_{\textbf{y}}} &~F^L_{\rho_t,{\bf x}_{\bf y}^t,\varpi^t,\tau}(	\textbf{x}_{\textbf{y}})\\
		\mathrm{s.t.}&~ G(	\textbf{x}_{\textbf{y}})\in \mathcal{K}.
	\end{aligned}\tag{L-MEAP-$\rho_t$-${\bf x}_{\bf y}^t$-$\varpi^t$-$\tau$}
\end{equation*}

The above linearized approximation problem is equivalent to a maximal monotone SVI so that PHM can be applied to solve it, as illustrated below. 
It is clear that this problem is convex.
By replacing $r_1(x)$ by ${\hat r}_{1,\rho_t,\tau_1}(x)$ and $r_{\xi_i}(y_i)$ by ${\hat r}_{\xi_i,\rho_t,\tau_2}(y_i)$, $i\in [K]$ in $\Psi$ of T-NNS-SCP (\ref{stage1})-(\ref{stage2}), and using a similar reformulation process as in Proposition \ref{prop3.1}, we can obtain that the problem (\ref{step2})
	is sure to have a solution and  a finite objective value. Furthermore, the  
    linearized approximation problem has a solution and the same finite objective value as   T-NNS-SCP with ${\hat r}_{1,\rho_t,\tau_1}(x)$ and ${\hat r}_{\xi_i,\rho_t,\tau_2}(y_i)$, $i\in [K]$.
The following theorem gives the relationship between  the solutions of the linearized approximation problem     and  a maximal monotone smooth two-stage SVI.
				
Let us define
	\begin{eqnarray}
    \label{lambdaL}
                      & & \lambda_{\rho_t,{\bf x}_{\bf y}^t,\varpi^t,\tau}^L({\bf x}_{\bf y}):= \left(
                      \begin{array}{c}
							\frac{1}{\rho_t} (U_1 x +u_1) - \frac{1}{\rho_t} \zeta_1^t +
							\tau_1 (x-x^t)\\
							p_i\left[\frac{1}{\rho_t} (U_{\xi_i} y_i + u_{\xi_i})- \frac{1}{\rho_t} \zeta_{\xi_i}^t\right] + p_i \tau_2 (y_i- y_i^t); \forall i\in [K]
						\end{array}
						\right),\\
						\label{HL}
					& &	H^{L}_{\rho_t,{\bf x}_{\bf y}^t,\varpi^t,\tau}(	\textbf{z}): =\left(			
						\begin{array}{c}
						\nabla \hat{f}(\textbf{x}_{\textbf{y}})	+\nabla G({\bf x}_{\bf y})M   + \nabla {\cal U}({\bf x}_{\bf y}) \lambda^L_{\rho_t,{\bf x}_{\bf y}^t,\varpi^t,\tau}({\bf x}_{\bf y})
                            \\
							-G({\bf x}_{\bf y})
						\end{array}		
						\right).\quad\quad
					\end{eqnarray}
	\begin{theorem}
	\label{lemma1}The problem (\ref{step2})  is sure to have a solution.  Moreover, 	the point {\rm  $\bar{\textbf{x}}_{\textbf{y}}$} is an  optimal solution of  problem (\ref{step2}) {\rm if and only if 
	$\bar{\textbf{x}}_{\textbf{y}}$} is a solution of  the following  SVI	{\rm
		\begin{equation}\label{svi1}
		0\in\nabla F^L_{\rho_t,{\bf x}_{\bf y}^t,\varpi^t,\tau}( 	\textbf{x}_{\textbf{y}})+N_{\Phi}( 	\textbf{x}_{\textbf{y}}).
	\end{equation}}		
	Furthermore, 
	{\rm  $\bar{\textbf{x}}_{\textbf{y}}$ }is an optimal solution of the SVI {\rm(\ref{svi1})}
	if and only if	there exist 	 
		{\rm$\bar{\bf{z}}= (\bar{{\bf x}}_{\bf y};\bar{M})\in D$}		such that		{\rm $\bar{\bf{z}}$} is a solution of  the  two-stage SVI	{\rm		
		\begin{equation}\label{SVI22}
		0\in H^{L}_{\rho_t,{\bf x}_{\bf y}^t,\varpi^t,\tau}(	\textbf{z}) +N_{D}(	\textbf{z}).
	\end{equation}}				
	\end{theorem}
\begin{proof}
	Problem (\ref{step2}) is equivalent to the smooth and convex programming
    \begin{equation*}
	\begin{aligned}
	\min 
	&~F^L_{\rho_t,{\bf x}_{\bf y}^t,\varpi^t,\tau}(	\textbf{x}_{\textbf{y}}) \quad 
	{\rm s.t.}\quad	\textbf{x}_{\textbf{y}}\in \Phi,
	\end{aligned}
	\end{equation*} 
	By \cite[Theorem 6.12, pp. 207]{Variational},  $\bar{\textbf{x}}_{\textbf{y}}$ is a global optimal solution of problem (\ref{step2})  if and only if $\bar{\textbf{x}}_{\textbf{y}}$ is a solution of  the   SVI (\ref{svi1}). 
	Furthermore, by assumption A5, we know that the RCQ holds at any point {\rm $\bar{\textbf{x}}_{\textbf{y}} \in\Phi$ for the system $G({\bf x}_{\bf y})\in {\cal K}$}.
	Based on Lemma \ref{normalcone} and the KKT condition, we can get {\rm $\bar{\textbf{x}}_{\textbf{y}}$} is an optimal solution of  problem (\ref{step2})
		if and only if
	there exist 
	{\rm$\bar{\bf{z}}= (\bar{{\bf x}}_{\bf y};\bar{M})\in D$}		such that		{\rm $\bar{\bf{z}}$} is a solution of  the two-stage  SVI in  (\ref{SVI22}).
			\end{proof}
	\begin{lemma}\label{promonotone}
	The  two-stage SVI  in  (\ref{SVI22})
	is of maximal monotone type. 
\end{lemma}	
    \begin{proof}
          By using the fact that all the functions in the  constraints of T-NNS-SCP (\ref{stage1})-(\ref{stage2}) are affine, similarly to the proof of Theorem 3.2 and Corollary 3.1 in \cite{Lizhang}, we can prove the lemma without difficulty. 
  \end{proof}
            
Let us 
		define the natural residual
			function 
			\begin{equation}\label{HLnat}
				\begin{aligned}
					&H_{\rho_t,{\bf x}_{\bf y}^t,\varpi^t,\tau,D}^{L, \rm{nat}}(\textbf{z}):=\textbf{z}-{\bf Proj}_{D}\bigg(\textbf{z}-H_{\rho_t,{\bf x}_{\bf y}^t,\varpi^t,\tau}^{L}(\textbf{z})\bigg).\\
				\end{aligned}	
			\end{equation}
	By direct computation, we have
			\begin{eqnarray}\label{H0-Ht}
				\left\|H^L_{\rho_t,{\bf x}_{\bf y}^t,\varpi^t,0}({\bf{z}}) - H^L_{\rho_t,{\bf x}_{\bf y}^t,\varpi^t,\tau}({\bf{z}})\right\| = \tau_1 \left\|x-x^t\right\| + \tau_2 \sum_{i=1}^K p_i \left\| y_i - y_i^t\right\|.
			\end{eqnarray}
			By the nonexpansive property of the projection operator we find
			\begin{eqnarray}\label{tH0-tHt}
				& & \left\|H^{L,{\rm nat}}_{\rho_t,{\bf x}_{\bf y}^t,\varpi^t,\tau,D}({\bf z})-H^{L, {\rm nat}}_{\rho_t,{\bf x}_{\bf y}^t,\varpi^t,0,D}({\bf z}) \right\| \nonumber\\
				&=& \left\|{\bf Proj}_D \left({\bf z}- H^L_{\rho_t,{\bf x}_{\bf y}^t,\varpi^t,0}({\bf{z}})\right) - {\bf Proj}_D \left({\bf z}- H^L_{\rho_t,{\bf x}_{\bf y}^t,\varpi^t,\tau}({\bf{z}})\right)\right\| \nonumber \\
				&\le&  \left\| H^L_{\rho_t,{\bf x}_{\bf y}^t,\varpi^t,\tau}({\bf{z}}) - H^L_{\rho_t,{\bf x}_{\bf y}^t,\varpi^t,0}({\bf{z}})\right\| \nonumber \\
				&=& \tau_1 \left\|x-x^t\right\|+ \tau_2 \sum_{i=1}^K p_i \left\| y_i - y_i^t\right\|.
			\end{eqnarray}	
		 
	We are ready to  describe our 
			SDC  method for solving (\ref{new_prob}) 	
            in  Algorithm \ref{algorithm1}.

\begin{algorithm}
\caption{Successive DC approximation method (SDC-PHM) for solving (\ref{new_prob}) 
}\label{algorithm1}
\begin{algorithmic}[1]
\State \textbf{Step 0.} Pick $\eta = (\eta_1,\eta_2,\eta_3)>0$ and $\tau=(\tau_1,\tau_2)>0$, sequences of positive numbers $\rho_t$ with $\rho_{t}\downarrow 0$.  Initialize ${\bf{z}}^0 \in D$ and set $t=0$.

    \State \textbf{Step 1.} Successive PHM for approximately solving (MEAP) with $\rho = \rho_t$.
	Set ${\bf z}^{t,0} = {\bf z}^0$, and $l=0$.

    \State   	\quad \textbf {Step 1-1.} Pick $\zeta^{t,l}_1 \in {\rm prox}_{\rho_t P_1} (U_1 x^{t,l} + u_1)$ and $\zeta_{\xi_i}^{t,l} \in {\rm prox}_{\rho_t P_{\xi_i}}(U_{\xi_i} y_i^{t,l}+u_{\xi_i})$, $i\in [K]$.
				Compute 
				\begin{eqnarray*}
					\varpi_1^{t,l}=\frac{1}{\rho_t} U_1^T \zeta_1^{t,l} \in 
					\partial R_{1,\rho_t}(x^{t,l})
					\quad \mbox{and}\quad
					\varpi_{\xi_i}^{t,l} = \frac{1}{\rho_t} U_{\xi_i}^T \zeta_{\xi_i}^{t,l} \in \partial R_{\xi_i,\rho_t}(y_i^{t,l}), ~i\in [K].
				\end{eqnarray*}

		   \State		\quad {\textbf {Step 1-2.}} Use the PHM (Algorithm \ref{algorithm2})  to approximately solve the following two-stage SVI 
				\begin{equation}\label{L-M}
					0\in	H^{L}_{\rho_t,{\bf x}_{\bf y}^{t,l},\varpi^{t,l},\tau}(	\textbf{z})+N_{D}(	\textbf{z}),
				\end{equation}
				\quad and get $\tilde{\textbf{{z}}}$	satisfying 
				\begin{equation}\label{resd}
					\left\|H_{\rho_t,{\bf x}_{\bf y}^{t,l},\varpi^{t,l},\tau,D}^{L,{\rm nat}}(\tilde{\textbf{z}})\right\|
					\le \eta_1\rho_t,
				\end{equation} 
				\quad and 
				\begin{equation}\label{LKLK_0}
					\begin{aligned}
						F^L_{\rho_t,{\bf x}_{\bf y}^{t,l},\varpi^{t,l},\tau}(\tilde{\textbf{x}}_{\textbf{y}})&\leq F^L_{\rho_t,{\bf x}_{\bf y}^{t,l},\varpi^{t,l},\tau}(\textbf{x}^{t,l}_{\textbf{y}}) + \frac{\eta_2}{(l+1)^2}.
					\end{aligned}
				\end{equation}	
				
	   \State \quad {\bf Step 1-3.}
				Set		${\textbf{z}}^{t,l+1}:= \tilde{\textbf{z}}$. 
				If  $\tilde{\bf{z}} $ does not satisfy
				\begin{equation}\label{TC}
					\begin{aligned}  
						\tau_1 \left\|\tilde x-x^{t,l}\right\|+ \tau_2 \sum_{i=1}^K p_i \left\| \tilde y_i - y_i^{t,l}\right\|\leq \eta_3 \rho_t^2,
					\end{aligned}
				\end{equation}
				\quad set $l:=l+1$ and go to {\textbf{Step 1-1}}.  
				Otherwise, set 
				$l_t := l+1$ and  go to 
				{\textbf{Step 2}}.

          \State        	{\bf Step 2.}  Update ${\textbf{z}}^{t+1} := {\textbf{z}}^{t,l_t}$, $t:= t+1$ and go to {\textbf{Step 1}}. 
\end{algorithmic}
\end{algorithm}

       How to approximately solve the two-stage SVI in (\ref{L-M}) in Algorithm \ref{algorithm1} efficiently
            is an important issue.  The two-stage SVI have been studied in \cite{TK,RoWets}. The  PHM  developed in \cite{RoSun} can efficiently  solve the two-stage SVI, because it can use parallel computation when facing a large number of scenarios, and efficient algorithms, e.g., the extra gradient method \cite[Algorithm 12.1.9]{Finite-Dimensional}, the  semi-smooth Newton method 
            \cite{semi1,semi2}, or the smoothing method \cite{smooth1,smooth2} can be employed for solving the subproblems in PHM.
             
		Since PHM is a decomposition method and in each iterate, a SVI corresponding to the first stage and the second stage with respect to a single scenario is considered, we denote by 
        $${\cal K}_i, D_i,(x_i;y_i), M^i, {\bf z}_i, {\hat f}^i(x_i,y_i), {\cal U}^i(x_i,y_i), G^i(x_i,y_i),
        \lambda_{\rho_t,{\bf x}_{\bf y}^t, \varpi^t,\tau}^{i,L}(x_i,y_i), H^{i,L}_{\rho_t,{\bf x}_{\bf y}^t \varpi^t,\tau}({\bf z}_i)  
        $$
       the counterparts to
        $${\cal K}, D, {\bf x}_{\bf y}, M, {\bf z}, {\hat f}({\bf x}_{\bf y}), {\cal U}({\bf x}_{\bf y}), G({\bf x}_{\bf y}), \lambda_{\rho_t,{\bf x}_{\bf y}^t, \varpi^t,\tau}^L({\bf x}_{\bf y}), H_{\rho_t,{\bf x}_{\bf y}^t, \varpi^t,\tau}^L({\bf z}) $$ in Table \ref{tab:symbols}, (\ref{lambdaL}) and (\ref{HL}), respectively,  by deleting the parts that  correspond to $j\ne i, i\in [K]$. For instance, 
        \begin{eqnarray*}
{\cal K}_i:=\{0_{n_1}\} \times -C_1 \times \{0_{n_2}\} \times -C_{2,\xi_i},\quad 
D_i := \mathbb{R}^{m_1}\times \mathbb{R}^{m_2}\times 
({\cal K}_i)^{\circ},\ i\in[K].
        \end{eqnarray*}
   
        Let  
			${\bf{\Lambda}}_{i} := (u_i,c_{1,i},c_{2,i}) \in \mathbb{R}^{m_1} \times \mathbb{R}^{n_1} \times \mathbb{R}^{s_1}$, where the elements ${u}_i$, ${c}_{1,i}$ and ${c}_{2,i}$ have the role of Lagrangian multipliers corresponding to the nonanticipativity for the first-stage vectors ${x}$, ${\alpha}_1$ and ${\alpha}_2$, and define
		\begin{equation*}\label{Hi}
    C_2^i(\breve{\mathbf{\Lambda}}_{i}^{\nu}):=(\breve{u}_{i}^{{\nu}};0_{\{m_2\}};\breve{c}_{1,i}^{{\nu}};\breve{c}_{2,i}^{{\nu}};0_{\{n_2+s_2\}}).
		\end{equation*}
       We are ready to give the  framework of the progressive hedging method (PHM) to approximately solve (\ref{L-M}) in Algorithm \ref{algorithm2} below. 
\begin{algorithm}
\caption{PHM 
for approximately solving (\ref{L-M})
}\label{algorithm2}
\begin{algorithmic}[1]
\State 	Given initial points $\breve{\textbf{z}}^0\in\mathcal{N}$ and $\breve{\Lambda}^0\in\mathcal{M}$
			with $\sum_{i=1}^Kp_i	\breve{u}_{i}^0=0$, 
			$\sum_{i=1}^Kp_i	\breve{c}_{1,i}^0=0$, $\sum_{i=1}^Kp_i	\breve{c}_{2,i}^0=0$,	
			$	\breve{v}_i^0=	\breve{s}_{1,i}^0=	\breve{s}_{2,i}^0=0$, $i=1,\dots,K$. Choose a step size $\sigma>0$. Set $\nu=0$.

    \State 	{\bf Step 1.}\quad For $i\in [K]$, solve the variational inequality 
			\begin{equation}\label{SVI2}
				0\in 	H^{i,L}_{\rho_t,x^t_i,y_i^t,\varpi^t,\tau}(	\textbf{z}_i)
				+C_2^i(\breve{\mathbf{\Lambda}}_{i}^{\nu}) + 
				\sigma(\textbf{z}_{i}-\breve{\textbf{z}}_{i}^{\nu})+N_{D_i}(	\textbf{z}_i),
			\end{equation}	
				to obtain a solution $\hat{\textbf{z}}_{i}^{\nu}$.

      \State  	{\bf Step 2.}\quad For $i\in [K]$, let					
			\begin{equation*}
				\begin{aligned}
					&	\breve{x}^{\nu+1}= \sum_{i=1}^{K}p_i\hat{x}_{i}^{\nu}, ~  		\breve{\alpha}_1^{\nu+1} = \sum_{i=1}^Kp_i\hat{\alpha}_{1,i}^{\nu}, 
				~ \breve{\alpha}_2^{\nu+1} = \sum_{i=1}^K p_i\hat{\alpha}_{2,i}^{\nu}, 
					\\
					& \breve{y}_{i}^{\nu+1} = \hat{y}_{i}^{\nu},		
					~ \breve{\pi}_{1,i}^{\nu+1}= \hat{\pi}_{1,i}^\nu,   	
					~ \breve{\pi}_{2,i}^{\nu+1} = \hat{\pi}_{2,i}^\nu, ~\breve{u}_{i}^{\nu+1} = \breve{u}_{i}^{\nu}+\sigma(\hat{x}_{i}^{\nu}-	\breve{x}^{\nu+1}), 
			 \\
					& \breve{c}_{1,i}^{\nu+1}= \breve{c}_{1,i}^{\nu}+\sigma(\hat{\alpha}_{1,i}^{\nu}-\breve{\alpha}_1^{\nu+1}), 
				~ 
					{\breve{c}}_{2,i}^{\nu+1} = 	\breve{c}_{2,i}^{\nu}+\sigma(\hat{\alpha}_{2,i}^{\nu}-	\breve{\alpha}_2^{\nu+1}).
					\\			
				\end{aligned}
			\end{equation*}		
			
			If (\ref{resd}) and (\ref{LKLK_0}) are satisfied at $\breve{\textbf{z}}$, then stop and set $\tilde{\textbf{z}}:=\breve{\textbf{z}}$.		Else,	
		set $\nu:=\nu+1$, and go to {\bf Step 1}. 
\end{algorithmic}
\end{algorithm}

\newpage
	\subsection{Theoretical guarantee for global convergence}
In this subsection, we first show that the SDC in Algorithm \ref{algorithm1},
where the two-stage SVI \eqref{L-M} is approximately solved by the PHM method in Algorithm \ref{algorithm2},
is well-defined in Theorem \ref{well-defined} under mild assumption {\bf A8} or {\bf A9} below.

	
	\begin{description}
		\item[{\bf A8.}] 
		The function $c(\cdot)$ is bounded from below on $ \mathbb{R}^{m_1}$  and the function $q_{\xi_i}(\cdot)$  is bounded from  below on  $ \mathbb{R}^{m_2}$ for each $i\in [K]$. 
		
		\item[{\bf A9.}] The set ${\{\bf{x}_{\bf y}}
		\ :\ 
		\hat f ({\bf x}_{\bf y}) 
		\le \upsilon\}$ is bounded for any $\upsilon \in \mathbb{R}$. 
	\end{description}
	
	\begin{theorem}\label{well-defined}
		Suppose that the successive PHM in {\bf{Step 1}} of Algorithm \ref{algorithm1} is   applied to approximately solve
     (MEAP) with $\rho = \rho_t$  
		in the $(t+1)$-th iteration of SDC. 
		Then the following statements hold.
		
		(i) For each fixed $l$
		and the corresponding two-stage SVI in {\bf Step 1-2} of Algorithm \ref{algorithm1}, 
        the criteria 
		{\rm (\ref{resd})} and {\rm(\ref{LKLK_0})} can be  satisfied by PHM in Algorithm \ref{algorithm2} within finite iterations. 
		
		(ii)
		If, in addition,
	assumption {\bf A8} or {\bf A9} holds, then  
		the
		criterion {\rm(\ref{TC})} in {\bf Step 1-3} of Algorithm \ref{algorithm1} is satisfied after finite calls of {\bf{Step 1-1}} and {\bf{Step 1-2}}.
	\end{theorem}
	\begin{proof}			
		(i) Let $\{\check{\textbf{z}}^{\nu}\}$ be the sequence generated by Algorithm \ref{algorithm2}  for solving (\ref{L-M}). By the convergence result of PHM, as $\nu$ tends to infinity, $\check{\textbf{z}}^{\nu}$ converges to a point  $\check{\textbf{z}}^{*}$ that is a solution of the two-stage SVI in (\ref{L-M}). Therefore, 
		\begin{eqnarray*}
			\lim_{\nu \to 
				+\infty} \left\|H_{\rho_t,{\bf x}_{\bf y}^t,\varpi^t,\tau,D}^{L,{\rm{nat}}}(\check{\textbf{z}}^{\nu})\right\| = \left\|H_{\rho_t,{\bf x}_{\bf y}^t,\varpi^t,\tau,D}^{L,{\rm{nat}}}(\check{\textbf{z}}^*)\right\| = 0.
		\end{eqnarray*}
		Thus for the constant $\varepsilon_1:=\eta_1\rho_t>0$, there exists $\Gamma(\varepsilon_1)\in \mathbb{N}:=\{1,2,\ldots\}$ such that for all $\nu > \Gamma(\varepsilon_1)$
		$$\left\|H_{\rho_t,{\bf x}_{\bf y}^t,\varpi^t,\tau,D}^{L,{\rm{nat}}}(\check{\textbf{z}}^{\nu})\right\|\le \eta_1\rho_t.$$
		By Theorem \ref{lemma1}, we know that the block $\check{\textbf{x}}_{\bf{y}}^*$ of $\check{\bf{z}}^*$ is an optimal solution of (\ref{step2}).
		Because $	F^L_{\rho_t,{\bf x}_{\bf y}^{t,l},\varpi^{t,l},\tau}$ is continuous, we have
		\begin{equation}\label{limm}
			\lim_{\nu\rightarrow+\infty}F^L_{\rho_t,{\bf x}_{\bf y}^{t,l},\varpi^{t,l},\tau}(\check{\textbf{x}}^{\nu}_{\textbf{y}})=F^L_{\rho_t,{\bf x}_{\bf y}^{t,l},\varpi^{t,l},\tau}(\check{\textbf{x}}^*_{\textbf{y}})=\min_{	\textbf{x}_{\textbf{y}}\in\Phi}F^L_{\rho_t,{\bf x}_{\bf y}^{t,l},\varpi^{t,l},\tau}(	\textbf{x}_{\textbf{y}}).
		\end{equation} 	
		Then for the constant $\varepsilon_2:=\frac{\eta_2}{(l+1)^2}>0$, there exists  $\Gamma_1(\varepsilon_2)\in\mathbb{N}$, such that for any $\nu>\Gamma_1(\varepsilon_2)$,
		\begin{equation}\label{limmm}
			\left|F^L_{\rho_t,{\bf x}_{\bf y}^{t,l},\varpi^{t,l},\tau}(\check{\textbf{x}}^{\nu}_{\textbf{y}})-F^L_{\rho_t,{\bf x}_{\bf y}^{t,l},\varpi^{t,l},\tau}(\check{\textbf{x}}^*_{\textbf{y}})\right|\leq \varepsilon_2,
		\end{equation}
		and consequently  for any $\nu>\Gamma_1(\varepsilon_2)$,
		\begin{equation}
			\begin{aligned}
				F^L_{\rho_t,{\bf x}_{\bf y}^{t,l},\varpi^{t,l},\tau}(\check{\textbf{x}}^{\nu}_{\textbf{y}})\leq F^L_{\rho_t,{\bf x}_{\bf y}^{t,l},\varpi^{t,l},\tau}(\check{\textbf{x}}^{*}_{\textbf{y}})+  \varepsilon_2
				\leq F^L_{\rho_t,{\bf x}_{\bf y}^{t,l},\varpi^{t,l},\tau}(\textbf{x}^{t,l}_{\textbf{y}})+  \varepsilon_2.
			\end{aligned}
		\end{equation}
		It is clear that for $\bar{\nu} = \max\{\Gamma(\varepsilon_1),\Gamma_1(\varepsilon_2)\}+1$, $\tilde{\bf z} = \check{\bf z}^{\bar \nu}$ satisfies both (\ref{resd}) and (\ref{LKLK_0}). This guarantees that the conditions \eqref{resd} and \eqref{LKLK_0}
are satisfied by Algorithm \ref{algorithm2} after finitely many iterations.
        
		(ii) 
	By (\ref{ineq1}) with ${\bf{x}}_{\bf y}^{t,l+1}$,  (\ref{LKLK_0}) 
	of Algorithm \ref{algorithm1},
	and (\ref{ineq2}), we get    
\begin{eqnarray}\label{Frhot-sequence}
			& & F_{\rho_t}(\textbf{x}^{t,l+1}_{\textbf{y}})	\nonumber\\
			&\leq& F^L_{\rho_t,{\bf x}_{\bf y}^{t,l},\varpi^{t,l},\tau}(\textbf{x}^{t,l+1}_{\textbf{y}})-\frac{\tau_1}{2} \left\|x^{t,l+1} - x^{t,l}\right\|^2 - \frac{\tau_2}{2} \sum_{i=1}^K p_i \left\|y_i^{t,l+1} - y_i^{t,l}\right\|^2 \nonumber\\
			&\leq& F^L_{\rho_t,{\bf x}_{\bf y}^{t,l},\varpi^{t,l},\tau}(\textbf{x}^{t,l}_{\textbf{y}})+ \frac{\eta_2}{(l+1)^2}-\frac{\tau_1}{2} \left\|x^{t,l+1} - x^{t,l}\right\|^2 - \frac{\tau_2}{2} \sum_{i=1}^K p_i \left\|y_i^{t,l+1} - y_i^{t,l}\right\|^2 \nonumber\\
			&=& F_{\rho_t}(\textbf{x}^{t,l}_{\textbf{y}})-\frac{\tau_1}{2} \left\|x^{t,l+1} - x^{t,l}\right\|^2 - \frac{\tau_2}{2} \sum_{i=1}^K p_i \left\|y_i^{t,l+1} - y_i^{t,l}\right\|^2+ \frac{\eta_2}{(l+1)^2}.
		\end{eqnarray}	
	
	By summing up $l$ from 0 to $\hat L$ of the above inequalities, letting $\hat L$ tend to $+\infty$, and using the fact $\sum\limits_{j=0}^{+\infty}\frac{\eta_2}{(j+1)^2}=\frac{\pi^2}{6}\eta_2$, 
    we obtain	
	\begin{eqnarray}\label{leftt}
		& & \sum_{l=0}^{+\infty}\frac{\tau_1}{2}\left\|x^{t,l+1}-x^{t,l}\right\|^{2} + \frac{\tau_2}{2} \sum_{l=0}^{+\infty} \sum_{i=1}^{K} p_i \left\|y_i^{t,l+1} - y_i^{t,l}\right\|^2 +\liminf_{\hat L\to  +\infty}F_{\rho_t}(\textbf{x}^{t,\hat L+1}_{\textbf{y}}) \nonumber\\
		&\leq&  F_{\rho_t}(\textbf{x}^{t,0}_{\textbf{y}})+	\sum_{j=0}^{+\infty}\frac{\eta_2}{(j+1)^2} < +\infty.
	\end{eqnarray}
	
	Next, we show that $\liminf_{{\hat L} \to + \infty} F_{\rho_t}({\bf x}_{\bf y}^{t, {\hat L}+1}) $ is bounded from below, under either assumption {\bf A8} or {\bf A9}. Using (\ref{rho-rhot}), we have for any $\rho>0$,
		\begin{eqnarray}\label{Frho_lower}
			F_{\rho}({\bf{x}}_{\bf{y}})
			= {\hat f}({\bf x}_{\bf y}) + r_{1,\rho}(x) + \sum_{i=1}^K p_i r_{\xi_i,\rho}(y_i),
		\end{eqnarray}
		and there exists $\underline{r}(\rho)\in \mathbb{R}$ such that 
		$$\inf_{{\bf x}_{\bf y}
			\in \mathbb{R}^{m_1+ m_2 K} 
		} \left\{r_{1,\rho}(x)+\sum_{i=1}^K p_i r_{\xi_i,\rho}(y_i)\right\}\ge \underline{r}(\rho),$$
		by assumption {\bf A1}, (\ref{r-r}), and (\ref{rho-rhot}) with $P_0$ being replaced by $P_1$ and $P_{\xi_i}$, respectively.
	
Thus, if assumption {\bf A8} holds, we know ${\hat f}({\bf x}_{\bf y})$ is also bounded from below and consequently ${\liminf_{\hat L\to  +\infty}F_{\rho_t}(\textbf{x}^{t,\hat L+1}_{\textbf{y}})}$ is bounded from below.
Furthermore, by (\ref{Frhot-sequence}), $\sum_{j=0}^{+\infty} \frac{\eta_2}{(j+1)^2}= \frac{\pi^2}{6}\eta_2$, and (\ref{Frho_lower}), we find that for any $l\ge 0$, 
		\begin{eqnarray}
			{\hat f}({\bf x}_{\bf y}^{t,l+1}) + \underline{r}(\rho_t) \le F_{\rho_t}({\bf{x}}_{\bf{y}}^{t,l+1}) \le F_{\rho_t}({\bf x}_{\bf{y}}^{t,0} )+ \frac{\pi^2}{6}\eta_2.
		\end{eqnarray}
		Hence, if assumption {\bf A9} holds, then we have $\{{\bf x}_{\bf y}^{t,l+1}\}_{l\ge 0}$ is bounded, and consequently we find ${\liminf_{\hat L\to  +\infty}F_{\rho_t}(\textbf{x}^{t,\hat L+1}_{\textbf{y}})}$ is also lower bounded.
	
Therefore,  under either assumption {\bf A8} or assumption {\bf A9}, (\ref{leftt}) indicates that   
	\begin{equation}\label{6000}
		\lim_{l\rightarrow+\infty} \left\|{\bf x}_{\bf y}^{t,l+1}-{\bf x}_{\bf y}^{t,l}\right\|=0.
	\end{equation}	
	Then for the constant $\varepsilon_3:=  \eta_3\rho_t^2>0$,  there exists $\hat{L}_1(\varepsilon_3)\in\mathbb{N}$ such that 
	for any $l>\hat{L}_1(\varepsilon_3)$, 
	\begin{eqnarray*}
	\tau_1 \left\| x^{t,l+1}-x^{t,l}\right\|+ \tau_2 \sum_{i=1}^K p_i \left\| y_i^{t,l+1} - y_i^{t,l}\right\| \le \varepsilon_3.
	\end{eqnarray*}
	Therefore,  the
	criterion (\ref{TC}) in {\bf Step 1-3}  is satisfied after finite calls of \bf{Step 1-1} and \bf{Step 1-2}.
\end{proof}
\begin{remark}
In fact, any algorithm for solving the maximal monotone smooth two-stage SVI can be applied in {\bf{Step 1-2}} instead of the PHM method in Algorithm \ref{algorithm2}, as long as the algorithm has the convergence result that there exists an accumulation point ${\bf \check{z}}$ of the algorithm such that it is a solution of the two-stage SVI. That is,  there exists an infinite subsequence ${\cal N}^{\sharp}\subseteq \mathbb{N}$ such that
\begin{eqnarray*}
	\lim_{\nu \to + \infty,\ \nu \in {\cal N}^{\sharp}}
	 {\bf z}^{\nu} = {\bf \check{z}}.
\end{eqnarray*}
 The well-definedness of Algorithm \ref{algorithm1}, i.e., the same results in Theorem \ref{well-defined}
  can be shown using  similar arguments, except that $\mathbb{N}$ should be replaced by ${\cal N}^{\sharp}$ in the proof for  statement (i).
\end{remark}
			
	Next, we will establish the convergence of the SDC method in Algorithm \ref{algorithm1}. We need the following assumption.
	
	\begin{description}
		\item[{\bf A10.}]  $P_1$ and $P_{\xi_i}$, $i\in [K]$ are  
			continuous functions. 
		\item [{\bf A11.}]
		$- \nabla {\cal U}({\bf x}_{\bf y}^*) \partial^{\infty} P({\cal U}({\bf x}_{\bf y}^*)) \cap N_{\Phi}({\bf x}_{\bf y}^*) = \{0\}.$
	\end{description}

\begin{theorem}[Convergence of SDC]\label{convergence} 
	Let {\rm $\{\textbf{z}^{t,l_t}\}$} be the sequence generated by the SDC method in Algorithm \ref{algorithm1}, and let {\rm $\textbf{x}^*_\textbf{y}$} be an arbitrary accumulation point of {\rm $\{\textbf{x}^{t,l_t}_\textbf{y}\}$}.
	Suppose that assumptions {\bf A5}, {\bf A7}, {\bf A10}, and {\bf A11} hold.
	Then	{\rm $\textbf{x}^*_\textbf{y}$} is a KKT point of {\rm(\ref{new_prob})}.
\end{theorem}
\begin{proof}
	Let $\cal T$ be an infinite subsequence of  $ \mathbb{N}$ such that 	
	$\lim_{t\in\mathcal{T}, ~t\rightarrow+\infty}\textbf{x}^{t,l_t}_\textbf{y}=\textbf{x}^*_\textbf{y}$. 
According to (\ref{tH0-tHt}), (\ref{resd}), and 
	(\ref{TC}) in Algorithm \ref{algorithm1}, 
	we have
	\begin{eqnarray}\label{main-inequality}
		& & \left\|H^{L,{\rm nat}}_{\rho_t,{\bf x}_{\bf y}^{t,l_t-1},\varpi^{t,l_t-1},0,D}({\bf{z}}^{t,l_t})\right\| \nonumber\\
		&\leq& \left\|H^{L,{\rm nat}}_{\rho_t,{\bf x}_{\bf y}^{t,l_t-1},\varpi^{t,l_t-1},\tau,D}({\bf z}^{t,l_t})-H^{L,{\rm nat}}_{\rho_t,{\bf x}_{\bf y}^{t,l_t-1},\varpi^{t,l_t-1},0,D}({\bf z}^{t,l_t})\right\|+ \left\|H^{L,{\rm nat}}_{\rho_t,{\bf x}_{\bf y}^{t,l_t-1},\varpi^{t,l_t-1},\tau,D}({\bf z}^{t,l_t})\right\| \nonumber \\
		&\leq& \tau_1  \left\|x^{t,l_t}-x^{t,l_t-1}\right\| 
		+ \tau_2 \sum_{i=1}^K p_i \left\|y_i^{t,l_t} - y_i^{t,l_t-1}\right\|
		+ \left\|H^{L,{\rm nat}}_{\rho_t,{\bf x}_{\bf y}^{t,l_t-1},\varpi^{t,l_t-1},\tau,D}({\bf z}^{t,l_t})\right\| \nonumber \\
		&\leq& \eta_3\rho_t^2+\eta_1\rho_t.
\end{eqnarray}	

	Letting $t\in {\cal T}$, $t\to +\infty$ in (\ref{main-inequality}), and using the formulas  (\ref{lambdaL})-(\ref{HL}), (\ref{HLnat}), we get
	\begin{eqnarray}\label{eq1}
		\lim_{t\in {\cal T}, \ t\to +\infty} \left[\nabla \hat f({\bf x}_{\bf y}^*) +  \nabla G({\bf x}_{\bf y}^*) M^{t,l_t} + \nabla {\cal U}({\bf x}_{\bf y}^*) \lambda^{t,l_t-1}\right]= 0,\\
		\lim_{t\in {\cal T}, \ t\to +\infty} \left[M^{t,l_t} - {\bf Proj}_{{\cal K}^{\circ}}(M^{t,l_t} + G({\bf x}_{\bf y}^*))\right] = 0, \label{eq2}
	\end{eqnarray}
	where  
	\begin{eqnarray*}
		M^{t,l_t}&:=&(\alpha_1^{t,l_t};\alpha_2^{t,l_t};\pi_{1,1}^{t,l_t};\ldots;\pi_{1,K}^{t,l_t};\pi_{2,1}^{t,l_t};\ldots;\pi_{2,K}^{t,l_t}),\\
		\lambda^{t,l_t-1}&:=& \left(
		\begin{array}{c}
			\lambda^{t,{l_t-1}}_1:=\frac{1}{\rho_t}(U_1 x^{t,{l_t}} + u_1) -  \frac{1}{\rho_t} \zeta^{t,l_t-1}_1\\
			\lambda^{t,{l_t-1}}_{\xi_i}:= p_i \left[\frac{1}{\rho_t}(U_{\xi_i} y_i^{t,l_t} + u_{\xi_i}) - \frac{1}{\rho_t} \zeta_{\xi_i}^{t,l_t-1}\right]; \forall i\in [K]
		\end{array}
		\right).
\end{eqnarray*}

Let 
	\begin{eqnarray}
		\delta^{t,l_t} = \left(\begin{array}{c}
			\delta^{t,l_t}_1:= \frac{1}{\rho_t} U_1 (x^{t,l_t} - x^{t,l_t-1})\\
			\delta^{t,l_t}_{\xi_i}:=   p_i \frac{1}{\rho_t} U_{\xi_i}(y_i^{t,l_t} - y_i^{t,l_t-1}); \forall i\in [K]
		\end{array}
		\right).      
	\end{eqnarray}
	According to (\ref{TC}), we can deduce that $\{\delta^{t,l_t}\}$ is bounded and 
	\begin{eqnarray}\label{d=0}
		\lim_{t\in {\cal T}, \ t\to +\infty} \delta^{t,l_t} = 0.
\end{eqnarray}
In view of $\zeta^{t,l_t-1}_1 \in \operatorname{prox}_{\rho_t P_1}(U_1 x^{t,l_t-1}+u_1)$,
$\zeta^{t,l_t-1}_{\xi_i} \in \operatorname{prox}_{\rho_t P_{\xi_i}}(U_{\xi_i} y_i^{t,l_t-1}+u_{\xi_i})$,
and Lemma~\ref{subgradient}, it follows that
	\begin{eqnarray*}\label{S1limit}
		 & & 
		\lambda^{t,l_t-1}_1-\delta^{t,l_t}_1 = 
		\frac{1}{\rho_t}(U_1 x^{t,l_t-1} + u_1) - \frac{1}{\rho_t} \zeta^{t,l_t-1}_1 \in \partial P_1(\zeta^{t,l_t-1}_1),\\
	& &	\lambda^{t,l_t-1}_{\xi_i} - \delta^{t,l_t-1}_{\xi_i} = p_i\left[\frac{1}{\rho_t}(U_{\xi_i} y_i^{t,l_t-1} + u_{\xi_i}) - \frac{1}{\rho_t} \zeta_{\xi_i}^{t,l_t-1}\right] \in \partial \left(p_i P_{\xi_i}(\zeta_{\xi_i}^{t,l_t-1})\right),\ i\in [K].
	\end{eqnarray*}
		The above two inclusions, together with the definition of $P({\cal U}({\bf x}_{\bf y}))$ in Table  \ref{tab:symbols} , yield
	\begin{eqnarray}\label{l-d-P}
			\lambda^{t,l_t-1} - \delta^{t,l_t}
			\in 
            \partial P(\zeta^{t,l_t-1}).
		\end{eqnarray}
	
	Next, we claim that
	\begin{eqnarray}\label{Gamma bounded}
		\{\Gamma_{t,l_t}:=\|M^{t,l_t}\| + \|\lambda^{t,l_t-1}\|\}_{t\in {\cal T}}\ \text{is bounded}.
	\end{eqnarray}
	Otherwise, subtracting $\nabla {\cal U}({\bf x}_{\bf y}^*) \delta^{t,l_t}$ from both sides of (\ref{eq1}), and then dividing  
	by $\Gamma_{t,l_t}$ and   letting $t\in{\cal T}$ tend to $+\infty$,
   we have 
	\begin{eqnarray}\label{neweq2}
		\lim_{t\in {\cal T}, \ t\to +\infty} \nabla G({\bf x}_{\bf y}^*) \frac{M^{t,l_t}}{\Gamma_{t,l_t}} + \nabla {\cal U}({\bf x}_{\bf y}^*) \frac{\lambda^{t,l_t-1} - \delta^{t,l_t}}{\Gamma_{t,l_t}} = 0.
	\end{eqnarray}
	Since $\left\{\frac{M^{t,l_t}}{\Gamma_{t,l_t}}\right\}_{t\in \cal T}$ and $\left\{\frac{\lambda^{t,l_t-1} - \delta^{t,l_t}}{\Gamma_{t,l_t}}\right\}_{t\in \cal T}$
	are bounded, there exists ${\cal T}_1 \subseteq {\cal T}$ such that  \begin{eqnarray}\label{eq3}
		\lim_{t\in {\cal T}_1, \ t\to + \infty} \Gamma_{t,l_t} = +\infty,\  \lim_{t\in {\cal T}_1, \ t\to + \infty}\frac{M^{t,l_t}}{\Gamma_{t,l_t}}= \tilde M,\ 
		\lim_{t\in {\cal T}_1, \ t\to +\infty}\frac{\lambda^{t,l_{t}-1} -\delta^{t,l_t}}{\Gamma_{t,l_t}} = \tilde \lambda.
	\end{eqnarray}
	It is easy to see that 
	$\|\tilde M\| + \|\tilde \lambda\| = 1.$  Moreover, by dividing (\ref{neweq2}) by $\Gamma_{t,l_t}$ and letting $t\in {\cal T}_1$, $t\to +\infty$, we find 
		\begin{eqnarray}\label{MG}
			\tilde M = (\tilde{\alpha}_1;\tilde{\alpha}_2;\tilde{\pi}_{1,1};\ldots;\tilde{\pi}_{1,K};\tilde{\pi}_{2,1};\ldots;\tilde{\pi}_{2,K})\in {\cal K}^{\circ},\ G({\bf x}_{\bf y}^*) \in {\cal K},\ {\tilde M^T} G({\bf x}_{\bf y}^*)=0. 
	\end{eqnarray}
That is, 
$$\tilde M \in N_{\cal K}(G(x^*)).$$

Since ${\cal U}({\bf x}_{\bf y}^{t,l_t-1}) \to {\cal U}({\bf x}_{\bf y}^*)$, $\zeta^{t,l_t-1} \in {\rm prox}_{\rho_t P}({\cal U}({\bf x}_{\bf y}^{t,l_t-1}))$, and assumption A1 that guarantees $P$ is lower bounded, it follows from Lemma~\ref{prox} that
$\zeta^{t,l_t-1} \to {\cal U}({\bf x}_{\bf y}^*)$. 
Thus we also have              
    \begin{eqnarray}
    \label{partial infinity}\tilde \lambda \in \partial^{\infty} P({\cal U}({\bf x}_{\bf y}^*)),
	\end{eqnarray}
	which is followed by (\ref{l-d-P}),
    $\frac{1}{\Gamma_{t,l_t}} \downarrow 0$, the last equation in (\ref{eq3}), 
    $
    \zeta^{t,l_t-1} \xrightarrow{P} 
    {\cal U}({\bf x}_{\bf y}^*)$  by assumption A10, and the 
	outer semi-continuity of $\partial^{\infty} P$. 
	
	There are only the following three possible cases for $\{\Gamma_{t,l_t}\}_{t\in {\cal T}}$ being unbounded.
	$$	\text{(i)}\	 \tilde \lambda = 0\ \text{and}\ \|\tilde M\|=1;\quad
		\text{(ii)}\	\tilde M = 0\ \text{and}\ \|\tilde \lambda\|=1;\quad
		\text{(iii)}\	\tilde M \ne 0\ \text{and}\ \tilde \lambda \ne 0.
$$
		Below we will show that all the three cases are indeed impossible.

	\vspace{0.3cm}
	If (i) holds, then 
	(\ref{neweq2})-(\ref{eq3}) imply
	\begin{equation}\label{8788}
		\begin{aligned}
			0=	&~	\nabla G^1(\textbf{x}^*_\textbf{y})\tilde{\alpha}_1+\nabla G^2(\textbf{x}^*_\textbf{y})\tilde{\alpha}_2 +\sum_{i=1}^K \bigg(\nabla G^{1,\xi_i}(\textbf{x}^*_\textbf{y})\tilde{\pi}_{1,i}+\nabla G^{2,\xi_i}(\textbf{x}^*_\textbf{y})\tilde{\pi}_{2,i}\bigg).
		\end{aligned}
	\end{equation} 
	Because the  RCQ holds at $\textbf{x}^*_\textbf{y}$  for the system $G({\bf x}_{\bf y}) \in {\cal K}$  by assumption {\bf A5}, there exists 
	$\bar{d}\in\mathbb{R}^{m_1+m_2K}$ satisfying (\ref{RCQ}).	
	Now we consider the only two possible cases in the following. 
	
	Case 1. When	$\tilde{\alpha}_2=0$ and $\tilde{\pi}_{2,i}=0$ for all 
	$i\in[K]$, (\ref{8788}) becomes
	\begin{equation}\label{hghg}
		\begin{aligned}
			0=	&~	\nabla G^1(\textbf{x}^*_\textbf{y})\tilde{\alpha}_1+\sum_{i=1}^K \nabla G^{1,\xi_i}(\textbf{x}^*_\textbf{y})\tilde{\pi}_{1,i},
		\end{aligned}
	\end{equation}
	with $\|(\tilde{\alpha}_1;\tilde{\pi}_{1,1};\ldots;\tilde{\pi}_{1,K})\|=1$ which contradicts the fact that the gradients $\nabla (G^1(\textbf{x}^*_\textbf{y}))_{j_1}$ and $\nabla (G^{1,\xi_i}(\textbf{x}^*_\textbf{y}))_{j_2}$,
	$i\in[K]$, $j_1\in[n_1]$, $j_2\in[n_2]$ are linearly independent based on the RCQ  at $\textbf{x}^*_\textbf{y}$ for the system $G({\bf x}_{\bf y})\in {\cal K}$. 
	
	Case 2. Otherwise,	$\tilde{\alpha}_2\neq0$ or $(\tilde{\pi}_{2,1};\ldots;\tilde{\pi}_{2,K})\neq0$. Because the RCQ holds at $\textbf{x}^*_\textbf{y}$ for the system  $G({\bf x}_{\bf y})\in {\cal K}$, according to the definition of RCQ, we know that there exists a vector $\bar d \in \mathbb{R}^{m_1 + m_2 K}$ such that 
	\begin{equation}
		\tilde{\alpha}_1^{T}\mathcal{J} G^1(\textbf{x}^*_\textbf{y})\bar{d}=0,~\tilde{\pi}_{1,i}^{T}\mathcal{J} G^{1,\xi_i}(\textbf{x}^*_\textbf{y})\bar{d}=0,~i\in[K].\nonumber
	\end{equation} Using these facts,  and  taking the inner product with the direction 
	$\bar{d}$ with both sides of (\ref{8788}), we obtain 
		\begin{equation}\label{eqqqq}
			\begin{aligned}
				0=	&~	\tilde{\alpha}_2^{T}\mathcal{J} G^2(\textbf{x}^*_\textbf{y})\bar{d}  +\sum_{i=1}^K \tilde{\pi}_{2,i}^{T}\mathcal{J} G^{2,\xi_i}(\textbf{x}^*_\textbf{y})\bar{d}.
			\end{aligned}
		\end{equation}		
		From (\ref{MG}), we have
		\begin{eqnarray}\label{complementarity}
			\tilde{\alpha}_2^{T}G^2(\textbf{x}^*_\textbf{y})=0 \quad \mbox{and}\quad \tilde{\pi}_{2,i}^T G^{2,\xi_i}
			({\bf x}_{\bf{y}}^*)=0, \ i\in [K].
	\end{eqnarray} 
	If 	$\tilde{\alpha}_2\neq0$,  using the first equation in (\ref{complementarity}) we have   
	\begin{equation}
		\tilde{\alpha}_2^{T}\mathcal{J} G^2(\textbf{x}^*_\textbf{y})\bar{d} =\tilde{\alpha}_2^{T}\bigg(G^2(\textbf{x}^*_\textbf{y})+\mathcal{J} G^2(\textbf{x}^*_\textbf{y})\bar{d}\bigg).\nonumber
	\end{equation}
	Because the RCQ holds at $\textbf{x}^*_\textbf{y}$ for the system $G({\bf x}_{\bf y}) \in {\cal K}$, we have
	\begin{equation}
G^2(\textbf{x}^*_\textbf{y})+\mathcal{J}G^2(\textbf{x}^*_\textbf{y})\bar{d}\in{\rm int}(-C_1).\nonumber
	\end{equation}
	This, together with $\tilde{\alpha}_2\in-C_1^{\circ}$, yields 
	\begin{equation}\label{fff}
		\tilde{\alpha}_2^{T}\mathcal{J} G^2(\textbf{x}^*_\textbf{y})\bar{d} =\tilde{\alpha}_2^{T}\bigg(G^2(\textbf{x}^*_\textbf{y})+\mathcal{J} G^2(\textbf{x}^*_\textbf{y})\bar{d}\bigg)<0.
	\end{equation}	
	Otherwise $(\tilde{\pi}_{2,1};\ldots;\tilde{\pi}_{2,K})\neq 0$. Then there exists $\bar i \in [K]$ such that $\tilde{\pi}_{2,\bar{i}}\neq0$. Using the second equation in (\ref{complementarity}), $\tilde \pi_{2,i} \in -C_{2,\xi_i}^{\circ}$, and $G^{2,\xi_i}(\textbf{x}^*_\textbf{y})+\mathcal{J} G^{2,\xi_i}(\textbf{x}^*_\textbf{y})\bar{d} \in {\rm int}(-C_{2,\xi_i})$, we can get
	\begin{equation}\label{huhu}
		\tilde{\pi}_{2,i}^{T}\mathcal{J} G^{2,\xi_i}(\textbf{x}^*_\textbf{y})\bar{d}=\tilde{\pi}_{2,i}^{T}\bigg(G^{2,\xi_i}(\textbf{x}^*_\textbf{y})+\mathcal{J} G^{2,\xi_i}(\textbf{x}^*_\textbf{y})\bar{d}\bigg)\le 0,\quad \forall i\in [K],
	\end{equation}
	and 
	\begin{eqnarray}\label{ineq}
		\tilde{\pi}_{2,\bar i}^{T} 	\mathcal{J} G^{2,\xi_{\bar i}}(\textbf{x}^*_\textbf{y})\bar{d}=\tilde{\pi}_{2,\bar i}^{T}\bigg(G^{2,\xi_{\bar i}}(\textbf{x}^*_\textbf{y})+\mathcal{J} G^{2,\xi_{\bar i}}(\textbf{x}^*_\textbf{y})\bar{d}\bigg)< 0.	\end{eqnarray}
	In view	of (\ref{fff}), (\ref{huhu}) and (\ref{ineq}), it follows that
	\begin{equation}	\begin{aligned}
			\tilde{\alpha}_2^{T}\mathcal{J} G^2(\textbf{x}^*_\textbf{y})\bar{d} 
			+ \sum_{i=1}^K \tilde{\pi}_{2,i}^{T}\mathcal{J} G^{2,\xi_i}(\textbf{x}^*_\textbf{y})\bar{d} < 0,
		\end{aligned}
	\end{equation}
	which contradicts (\ref{eqqqq}). 
	
Therefore, the RCQ  at ${{\bf x}_{\bf y}^*}$ for the system $G({\bf x}_{\bf y}) \in {\cal K}$ guarantees that the situation in (i) does not exist.
	
	\vspace{0.2cm}
	
 If (ii) holds, then 
		(\ref{neweq2})-(\ref{partial infinity}) imply
		$$\nabla {\cal U}({\bf x}_{\bf y}^*) \tilde \lambda = 0, \quad \tilde \lambda \in \partial^{\infty} P({\cal U}({\bf x}_{\bf  y}^*)). $$
		However, by assumption {\bf A7}, we can deduce  $\tilde \lambda = 0$, which contradicts the fact $\|\tilde \lambda\|=1$. Hence, situation in (ii) does not exist also. 
	
	\vspace{0.2cm}
	
If (iii) holds, then  (\ref{neweq2})-(\ref{partial infinity}) yield 
		\begin{eqnarray*}
			\nabla G({\bf x}_{\bf y}^*) \tilde M = -\nabla {\cal U}({\bf x}_{\bf y}^*) \tilde \lambda,\ \tilde M \in N_{\cal K}(G({\bf x}_{\bf y}^*)),\ \tilde \lambda \in \partial^{\infty} P({\cal U}({\bf x}_{\bf y}^*)).
		\end{eqnarray*}
		According to (\ref{2-3}), we know that $\nabla G({\bf x}_{\bf y}^*) \tilde M \in N_{\Phi}({\bf x}_{\bf y}^*)$. Using $\tilde \lambda \in \partial^{\infty} P({\cal U}({\bf x}_{\bf y}^*)) $ and assumption {\bf A11}, we get  $\nabla  {\cal U}({\bf x}_{\bf y}^*)\tilde \lambda =0$ since $\tilde \lambda \in \partial^{\infty}P({\cal U}({\bf x}_{\bf y}^*))$.  We then derive $\tilde\lambda =0$ by assumption  {\bf A7}.  This contradicts the fact that $\tilde \lambda \ne 0$ in (iii). Thus the situation in (iii) can not exist either.
	
Till now, we have shown that (\ref{Gamma bounded}) holds, i.e., 
		the sequence	
		$\{\Gamma_{t,l_t}\}_{t\in \cal T}$ is bounded. It is then clear that the sequence $\{\textbf{z}^{t,l_t}\}_{t\in\mathcal{T}}$
	generated by Algorithm \ref{algorithm1} is bounded. 
	Thus there exist an infinite subsequence $\mathcal{T}_2\subseteq {\mathcal T}_1$ and vectors ${\bf z}^*$, $M^*$ and $\lambda^*$, such that 
	\begin{eqnarray*}
		\textbf{z}^* = \lim_{t\in\mathcal{T}_2, ~t\rightarrow+\infty}\textbf{z}^{t,l_t},\quad 
		M^* = \lim_{t\in\mathcal{T}_2, ~t\rightarrow+\infty} M^{t,l_t},
	\end{eqnarray*}
	and  
	\begin{eqnarray*}
		\lambda^* = \lim_{t\in\mathcal{T}_2, ~t\rightarrow+\infty} \lambda^{t,l_t-1} = 
		\lim_{t\in\mathcal{T}_2, ~t\rightarrow+\infty} (\lambda^{t,l_t -1} - \delta^{t,l_t})  \in \partial P({\cal U}({\bf x}_{\bf y}^*))
	\end{eqnarray*}
	by (\ref{d=0}),  (\ref{l-d-P}),  ${\bf x}_{\bf y}^{t,l_t-1} \xrightarrow{P({\cal U}(\cdot))} {\bf x}_{\bf y}^*$ as $t\in {\cal T}_2 \subseteq {\cal T}_1$, and the outer semicontinuity of $\partial P({\cal U}(\cdot))$.

		Using (\ref{eq1}), (\ref{eq2}), and the above three equations,
	we find that
\begin{eqnarray*}
	& & 0 = \nabla \hat f({\bf x}_{\bf y}^*) + \nabla G({\bf x}_{\bf y}^*) M^* + \nabla {\cal U}({\bf x}_{\bf y}^*)\lambda^*,\ \lambda^* \in \partial P({\cal U}({\bf x}_{\bf y}^*)),\\
	& & M^* = {\rm Proj}_{{\cal K}^{\circ}} (M^* + G({\bf x}_{\bf y}^*)) ,
\end{eqnarray*}
where the last equation is equivalent to $M^* \in N_{{\cal K}}(G({\bf x}_{\bf y}^*))$.
Therefore, there exist $\lambda^* \in \partial P({\cal U}({\bf x}_{\bf y}^*))$ and $M^* \in N_{\cal K} (G({\bf x}_{\bf y}^*))$ such that 
\begin{equation*}
	\begin{aligned}
		&	0 = \nabla \hat{f}(\textbf{x}^*_{\textbf{y}})+
		\nabla {\cal U}({\bf x}_{\bf y}^*) \lambda^*	
		+\nabla G(\textbf{x}^*_{\textbf{y}})M^*,
	\end{aligned}
\end{equation*}
and consequently
{\rm $\textbf{x}_\textbf{y}^*$} is  a KKT point of (\ref{new_prob}).	
\end{proof}

\begin{theorem}\label{l0convergence}
Suppose
$\Psi({\bf x}_{\bf y}) = \sum_{i=1}^{m_1+m_2 K}
\psi_i(({\bf x}_{\bf y})_i)$ for some lsc functions $\psi_i : \mathbb{R} \to \mathbb{R}$. Let ${\cal I} = \{i\ : \partial^{\infty} \psi_i(({\bf x}_{\bf y}^*)_i)=\{0\}\}$ and ${\cal I}^c$ be the complement of ${\cal I}$ with respect to $\{1,\ldots,m_1 + m_2 K\}$.
Assume further that for any $i\in {\cal I}^c$, $\partial \psi_i(({\bf x}_{\bf y}^*)_i) = \mathbb{R}$, and the RCQ holds at $({\bf x}_{\bf y}^*)_{\cal I}$ for the system $G(({\bf x}_{\bf y})_{\cal I}, ({\bf x}_{\bf y}^*)_{{\cal I}^c})\in {\cal K}$.
Let {\rm $\{\textbf{z}^{t,l_t}\}$} be the sequence generated by the SDC method in Algorithm \ref{algorithm1} and {\rm $\textbf{x}^*_\textbf{y}$} be an arbitrary  accumulation point of {\rm $\{\textbf{x}^{t,l_t}_\textbf{y}\}$}. 	
Then
{\rm $\textbf{x}^*_\textbf{y}$} is a KKT point of {\rm(\ref{new_prob})}. 
\end{theorem}
\begin{proof}
Using the separability of $\Psi$, we know that for any ${\bf x}_{\bf y}$, ${\cal U}({\bf x}_{\bf y}) = {\bf x}_{\bf y}$, 
\begin{eqnarray*}
	\partial P({\cal U}({\bf x}_{\bf y})) = \partial P({\bf x}_{\bf y}) = \partial \Psi({\bf x}_{\bf y}),
\end{eqnarray*} and  by \cite[Proposition 10.5]{Variational},
\begin{eqnarray*}
	\partial \Psi({\bf x}_{\bf y}) = \partial_{({\bf x}_{\bf y})_{\mathcal{I}}} \Psi({\bf x}_{\bf y}) \times \partial_{({\bf x}_{\bf y})_{\mathcal{I}^c}} \Psi({\bf x}_{\bf y}).
\end{eqnarray*}
By $\partial^{\infty} \psi_i(({\bf x}_{\bf y}^*)_i) = \{0\}$ for all $i\in {\cal I}$, 
we have 
\begin{eqnarray*}
	& & \partial^{\infty} \Psi_{\cal I}({\bf x}_{\bf y}^*) \cap N_{\Phi_{\cal I}} ({\bf x}_{\bf y}^*)= \{0\},\ \mbox{with}\ \Phi_{\cal I}:= \{{\bf x}_{\bf y}\ :\ G(({\bf x}_{\bf y})_{\cal I}, ({\bf x}_{\bf y}^*)_{{\cal I}^c}) \in {\cal K}\},\\
	& & \lambda \in \partial^{\infty} \Psi_{{\cal I}}({\bf x}_{\bf y}^*) \Longrightarrow \lambda =0.
\end{eqnarray*}
Similarly to (\ref{eq1})-(\ref{eq2}), Algorithm \ref{algorithm1}  guarantees that  
\begin{eqnarray*}
	\lim_{t\in {\cal T},\ t\to +\infty} \left[\nabla_{({\bf x}_{\bf y})_{\cal I}} {\hat f}({\bf x}_{\bf y}^*) + \nabla_{({\bf x}_{\bf y})_{\cal I}} G({\bf x}_{\bf y}^*) M^{t,l_t} + \lambda^{t,l_t-1}\right] = 0,\\
	\lim_{t\in {\cal T}, \ t\to +\infty} \left[M^{t,l_t} - {\bf Proj}_{{\cal K}^{\circ}}(M^{t,l_t}+G({\bf x}_{\bf y}^*))\right] = 0.
\end{eqnarray*}
Since the RCQ holds at $({\bf x}_{\bf y}^*)_{\cal I}$ for the system $G(({\bf x}_{\bf y})_{\cal I}, ({\bf x}_{\bf y}^*)_{{\cal I}^c}) \in {\cal K}$, using  arguments similar to  Theorem \ref{convergence}, and deducing from the above two equations, we can show that 
\begin{eqnarray*}
	0\in \nabla_{({\bf x}_{\bf y})_{\cal I}} {\hat f}({\bf x}_{\bf y}^*) + \partial_{({\bf x}_{\bf y})_{\cal I}} \Psi({\bf x}_{\bf y}^*) + \nabla_{({\bf x}_{\bf y})_{\cal I}} G({\bf x}_{\bf y}^*) N_{\cal K} (G({\bf x}_{\bf y}^*)).
\end{eqnarray*}
According to the assumption that $\partial \psi_{i}(({\bf x}_{\bf y}^*)_i) = \mathbb{R}$ for all $i\in {\cal I}^c$, it is obvious that 
\begin{equation*}
	0\in\nabla_{({\bf x}_{\bf y})_{\mathcal{I}^c}} \hat{f}(\textbf{x}^*_{\textbf{y}})+\partial_{({\bf x}_{\bf y})_{\mathcal{I}^c}}\Psi(\textbf{x}^*_{\textbf{y}}) +\nabla_{({\bf x}_{\bf y})_{\mathcal{I}^c}} G(\textbf{x}^*_{\textbf{y}})N_\mathcal{K}(G(\textbf{x}^*_{\textbf{y}})).
\end{equation*}
Combining the above two inclusions, we have  
\begin{eqnarray*}
	0\in \nabla {\hat f}({\bf x}_{\bf y}^*) + \partial \Psi({\bf x}_{\bf y}^*) + \nabla G({\bf x}_{\bf y}^*) N_{\cal K}( G({\bf x}_{\bf y}^*)).
\end{eqnarray*}
Therefore, ${\bf x}_{\bf y}^*$ is a KKT point of (\ref{new_prob}). 	
\end{proof}

\subsection{Useful observations and extensions}
\label{extensions}
In this subsection, we provide some useful observations and extensions. The idea of developing the SDC method and the roles of  assumptions are briefly illustrated in Figure \ref{relationship}.


\begin{figure}[h]  
    \centering     
    \includegraphics[width=1\textwidth]{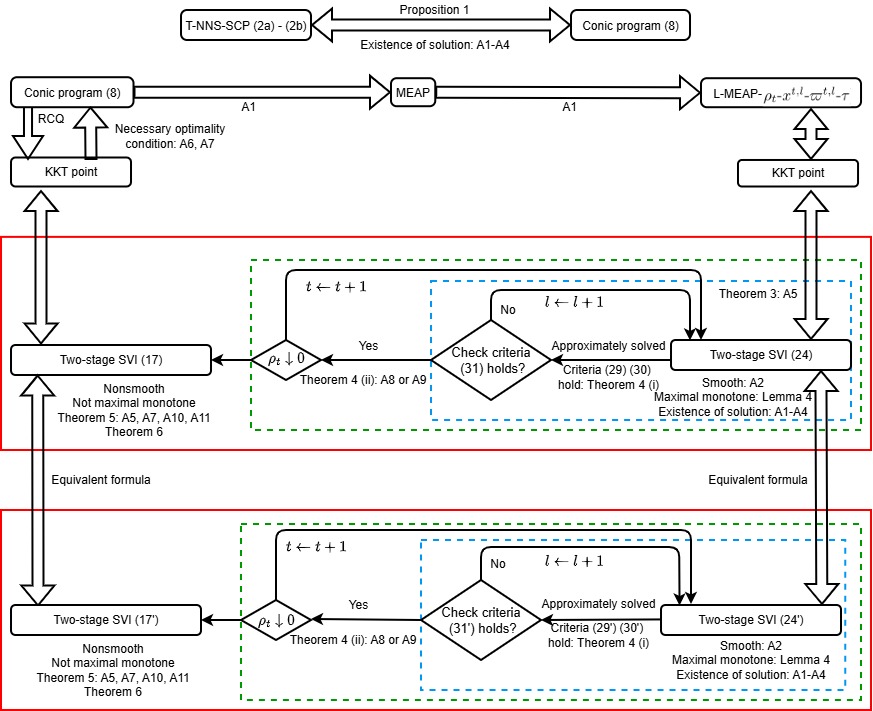} 
    \caption{Overall framework of the proposed  SDC-PHM and its theoretical justification} 
  	\label{relationship} 
\end{figure}

\newpage

\noindent{\bf Observations and extensions}

(I) The first two  rows above the first red box provide the idea of simplifying the original nonconvex nonsmooth T-NNS-SCP to a  convex smooth two-stage SP. Algorithm \ref{algorithm1} is illustrated in the first red box in Figure \ref{relationship}. That is, through the  KKT points, the nonsmooth T-NNS-SCP is solved via the SVI approach from the maximal monotone two-stage SVI. 

(II) In practice, the KKT point of (\ref{new_prob}) may have a different but equivalent formula, which leads to an equivalent two-stage SVI that is different from (\ref{OKKT}).
Adopting the same formula for the KKT point of (\ref{step2}), and the corresponding two-stage SVI, instead of (\ref{SVI22}), we can still use the same framework in Algorithm \ref{algorithm1}, and get the convergence results that any accumulation point is a solution of the two-stage SVI corresponding to the original problem, and provide a KKT point of (\ref{new_prob}).This extension is illustrated in the red box at the bottom of Figure \ref{relationship}. See Sect. 5 for an example of this extension. 

(III) The functions are all affine in conic constraints of  (\ref{stage1}) and (\ref{stage2}). This requirement is  only used in the proof of Lemma \ref{promonotone}, to guarantee that the two-stage SVI  in (\ref{SVI22}) corresponding to the L-MEAP is maximal monotone. Therefore, if there exist functions in conic constraints that are not affine but   the SVI in (\ref{SVI22}) is still maximal monotone and the existence of a solution is guaranteed, then Algorithm \ref{algorithm1} and the convergence results in Theorems  \ref{well-defined}-\ref{l0convergence} hold.  For this extension,  we write $\nabla G({\bf x}_{\bf y})$ instead of a constant matrix $\nabla G$ throughout this paper.

(IV) The natural residual function 
$H^{L,{\rm nat}}_{\rho_t, {\bf x}_{\bf y}^{t,l},{\varpi}^{t,l},\tau,D}(\tilde{\bf z})$
in the stopping criteria (\ref{resd}) can be replaced by any other residual function $H^{L,\rm{res}}_{\rho_t, {\bf x}_{\bf y}^{t,l},{\varpi}^{t,l},\tau,D}(\tilde{\bf z})$ as long as it satisfies
\begin{eqnarray}
& & H^{L, \rm{res}}_{\rho_t, {\bf x}_{\bf y}^{t,l},{\varpi}^{t,l},\tau,D}(\tilde{\bf z}) =0 
\Longleftrightarrow	\tilde{{\bf z}} \in {\rm SOL}\left(D, H^{L}_{\rho_t, {\bf x}_{\bf y}^{t,l},{\varpi}^{t,l},\tau,D}({\bf z})\right),
\end{eqnarray}
and 
\begin{eqnarray}
& &\left\|H^{L,\rm{res}}_{\rho_t, {\bf x}_{\bf y}^{t,l},{\varpi}^{t,l},\tau,D}(\tilde{\bf z})
-H^{L,\rm{res}}_{\rho_t, {\bf x}_{\bf y}^{t,l},{\varpi}^{t,l},0,D}(\tilde{\bf z})\right\|
\nonumber\\
&\le& \tau_1 \|x-x^{t,l}\| + \tau_2 \sum_{i=1}^K p_i \|y_i-y_i^{t,l}\|\quad 
\mbox{for some constants}\ \tau_1,\tau_2>0. 
\end{eqnarray}

(V) When only the first-stage/second-stage objective function involves a nonsmooth term, we can just omit the process in this paper for the nonsmooth term in the second-stage/first-stage objective  function accordingly.  Then Algorithm \ref{algorithm1} as well as the convergence results in Theorems  \ref{well-defined}-\ref{l0convergence} hold.

(VI) Assumption {\bf A5} requires the existence of a Slater point of T-NNS-SCP (\ref{stage1})-(\ref{stage2}), and hence the RCQ holds at any feasible point with respect to the system $G({\bf x}_{\bf y})\in {\cal K}$. Under assumption {\bf A5}, it is shown in Theorem \ref{lemma1} that the optimality condition of  the convex program (\ref{step2}) is equivalent to the smooth and maximal monotone two-stage SVI (\ref{SVI22}). If the two-stage SVI  (\ref{SVI22}) is of special structure, e.g., it is a P-matrix LCP \cite{CottlePang}, then we can still solve it without requiring that it is of  maximal monotone. 
In this case, we do not need {\bf A5}, but only require the RCQ holds at any accumulation point of $\{{\bf x}_{\bf y}^t\}$ to obtain the same convergence results as in this paper.

\section{Nonconvex  nonsmooth two-stage extension of Markowitz's mean-variance model}
\label{appl}

In this section, we  adopt the extension of Markowitz’s mean-variance model \cite{Sanjay} as our baseline. To better align the model with practical asset management requirements, 
we impose non-negativity constraints on portfolio weights to prohibit short selling, 
and incorporate an $\ell_0$-norm regularization term into the objective function to induce sparsity, leading 
to  the following nonconvex  nonsmooth two-stage SP:
\begin{equation}\label{orimodel1}
	\begin{aligned}
		\min_{x\in \mathbb{R}^{n}}\ \ & x^{T}Q_{1}x+ \gamma\|x\|_0+ \sum_{i=1}^{K}\frac{1}{K}\vartheta_{i}(x,\xi_{i})\\
		\mathrm{s.t.}~\ \   &e^{T}x=1, \\
		& \bar{r}_{1}^{T}x\geq \bar{r}^{\min}_{1},\\
		& x\geq 0,
	\end{aligned}
\end{equation}
where 
$\vartheta_{i}(x,\xi_{i})$, $i=1,...,K$ is the optimal value of the second-stage problem
\begin{equation}\label{orimodel2}
	\begin{array}{ll}
		\min\limits_{y_i\in \mathbb{R}^{n}} & y_i^{T}Q_{2,i}y_i\\
		~~\mathrm{s.t.}  &e^{T}y_i=1,\\ 
		& \bar{r}_{2,i}^{T}y_i\geq \bar{r}^{\min}_{2,i},\\
		& y_i\geq 0,\\
		& \|x-y_i\|\leq\tau_{2,i}.
	\end{array}
\end{equation}
Here,
\begin{itemize}
	\item $x$: vector of first-stage portfolio positions;
	\item $e$: vector of all ones;
	\item $r_{1}, \bar{r}_{1}, \bar{r}_{1}^{\min}$: vector of asset returns, expected asset returns, and lower bound for expected returns in the first period, respectively;
	\item $Q_{1}$: covariance of asset returns in the first period;
	\item $y_i$: vector of second-stage portfolio positions under scenario $i$;
	\item $r_{2,i}, \bar{r}_{2,i}, \bar{r}_{2,i}^{\min}$: vector of asset returns, expected asset returns, and minimum required return in the second period under scenario $i$, respectively;
	\item $Q_{2,i}$: covariance of asset returns in the second period under scenario $i$;
	\item $\tau_{2,i}$: upper bound for the distance of first-stage decision and second-stage decision.\end{itemize}

The constraint $\|x-y_i\| \le \tau_{2,i}$ can be considered as a second-order cone (SOC) constraint  by adding an auxiliary variable $t$, i.e., a linear function in the second-order cone
$$(x-y_i;t) \in \{(s_1;s_2)\in \mathbb{R}^{n+1} \ :\ \|s_1\|\le s_2\},$$
 together with an equality constraint $t=\tau_{2,i}$.  It can also  be equivalently transformed to a nonlinear constraint
 $$\tau_{2,i}^2 - \|x-y_i\|^2\ge 0,$$
 which is a nonlinear function in the nonnegative cone. Here we adopt the second one, for ease of computation.  Suppose assumption {\bf A5} holds. It is then easy to see that assumptions {\bf A1}-{\bf A4}, and {\bf A6} hold. Then  (\ref{orimodel1})-(\ref{orimodel2}) can be reformulated into the following equivalent two-stage SP:
\begin{equation}\label{CP-1}
	\begin{aligned}
		\min_{\textbf{x}_\textbf{y}} &~x^TQ_1x+\gamma\|x\|_0 +\sum_{i=1}^K\frac{1}{K}y_i^{T}Q_{2,i}y_i\\
		\mathrm{s.t.}  &~ e^{T}x-1 =0,   && \text{[Multiplier: } 
		\alpha_1 \in \mathbb{R} \text{]} \\  
		&~ \bar{r}_{1}^{T}x- \bar{r}^{\min}_{1} \geq 0, &&  
		 \text{[Multiplier: } \alpha_2^r \ge 0 \text{]} \\
		&~ x\geq 0, 	&& \text{[Multiplier: } \alpha^x \ge 0 \text{]}\\	
		&~ e^{T}y_i - 1 = 0,  ~i\in [K], 
			&& \text{[Multiplier: } \pi_{1,i} \in \mathbb{R} \text{]}\\	
		&~ \bar{r}_{2,i}^{T} y_i - \bar{r}^{\min}_{2,i} \geq 0,~i\in [K],
			&& \text{[Multiplier: } \pi_{2,i}^r \ge 0 \text{]}\\
		&~ y_i\geq 0,~i\in [K], 	
		&& \text{[Multiplier: } \pi_{2,i}^y \ge 0 \text{]}\\
		& ~\tau_{2,i}^2 - \|x-y_i\|^2\geq 0 ,~i\in [K], 
			&& \text{[Multiplier: } \pi_{2,i}^{\tau} \ge 0\text{]}.
	\end{aligned}
\end{equation}	

For this problem, only the first-stage involves nonconvex nonsmooth term constituted by the $\ell_0$-norm. By using Moreau envelope, performing linearization the second term of its DC formula, and adding a regularization term at the current iterate $x^{t,l}$, we obtain 
the objective function of (\ref{step2}) 
\begin{eqnarray*}
	F^L_{\rho_t,x^{t,l},\varpi^{t,l},\tau}(\textbf{x}_{\textbf{y}})
	&:=&x^TQ_1x +\sum_{i=1}^K\frac{1}{K}y_i^{T}Q_{2,i}y_i 
	+\frac{1}{2 \rho_t}\|x\|^2 
	-R_{\rho_t}(x^{t,l})-\langle x-x^{t,l}, \varpi^{t,l}\rangle \\
	& &\quad +\frac{\tau}{2}\|x - x^{t,l}\|^2,
\end{eqnarray*}
As derived in Lemma \ref{subgradient} and \cite{L0}, the subdifferential of $R_{\rho_t}(x^{t,l})$ is given by
\begin{equation*}
	(\partial R_{\rho_t}(x^{t,l}))_j=
	\begin{cases}
		\{0\}, & |(x^{t,l})_{j}| < \sqrt{2\gamma\rho_t }, \\
		\{\frac{(x^{t,l})_{j}}{\rho_t}\}, & |(x^{t,l})_{j}| > \sqrt{2\gamma\rho_t }, \\
		\left\{0, \frac{(x^{t,l})_{j}}{\rho_t}\right\}, & |(x^{t,l})_{j}| = \sqrt{2\gamma\rho_t },
	\end{cases} \quad j \in [n].
\end{equation*}
We select a specific element $\varpi^{t,l}\in \partial R_{\rho_t}(x^{t,l})$ as:
\begin{equation}\label{norm0}
	(\varpi^{t,l})_j=
	\begin{cases}
		0, & |(x^{t,l})_{j}| < \sqrt{2\gamma\rho_t }, \\
		\frac{(x^{t,l})_{j}}{\rho_t}, & |(x^{t,l})_{j}| \geq \sqrt{2\gamma\rho_t },
	\end{cases}
    \quad  j \in [n].
\end{equation}

Applying Theorem \ref{lemma1}, we find that the KKT condition of  (\ref{step2}) for (\ref{CP-1}) is equivalent to the following two-stage SVI, which is in fact a two-stage mixed nonlinear complementarity problem:
\begin{equation}\label{two_ncp}
	\left\{\begin{aligned}
		& 2Q_{1}x+\frac{1}{\rho_t}x-\varpi^{t,l}+\tau(x-x^{t,l})+\alpha_{1}e
		-\alpha_{2}^r\bar{r}_{1}- \alpha_{2}^x
		+\sum_{i=1}^{K}2\pi_{2,i}^{\tau}(x-y_i) = 0,\\
		& \frac{2}{K}Q_{2,i}y_i+\pi_{1,i}e-\pi_{2,i}^r\bar{r}_{2,i}- \pi_{2,i}^{y_i}+2\pi_{2,i}^{\tau}(y_i-x) = 0, \quad  i\in[K],\\
		& -e^{T}x+1=0, \\
		& 0\leq \left(\bar{r}_{1}^Tx-\bar{r}_{1}^{\min}\right) \perp \alpha_2^r \geq 0,\\
		& 0\leq x \perp \alpha_2^x \geq 0, \\
		& -e^{T}y_i+1=0,  i\in [K],\\
		& 0\leq \left(\bar{r}_{2,i}^{T}y_i-\bar{r}_{2,i}^{\min}\right) 
		\perp \pi_{2,i}^r\geq 0, \quad i\in [K],\\
		& 0\leq y_i \perp \pi_{2,i}^{y_i} \geq 0, \quad  i\in [K],\\
		& 0\leq \left(\tau_{2,i}^{2}-\|x-y_i\|^{2}\right) \perp 
		\pi_{2,i}^{\tau}\geq 0, \quad \forall i\in [K].
	\end{aligned}
	\right.
\end{equation}
Here $d_1 \perp d_2$ refers to $d_1^T d_2 =0$.
It is easy to see that the above two-stage SVI is equivalent to  the following two-stage SVI in a concise form:
\begin{eqnarray}\label{two_ncp}
	\left\{\begin{array}{l}
		0\leq\left(2Q_{1}x+\frac{1}{\rho_t}x-\varpi^{t,l}+\tau(x-x^{t,l})
		+\alpha_{1}e-\alpha_{2}^r\bar{r}_{1}
		+\sum_{i=1}^{K}2\pi_{2,i}^\tau(x-y_i)\right) \perp x \geq 0,\\
		0\leq \left(	\frac{2}{K}Q_{2,i}y_i+\pi_{1,i}e-\pi_{2,i}^r\bar{r}_{2,i}
		+2\pi_{2,i}^\tau(y_i-x)\right) \perp y_i\geq 0,\ i\in[K],\\
		-e^{T}x+1=0,\\
		0\leq \left(\bar{r}_{1}^Tx-\bar{r}_{1}^{\min}\right) \perp \alpha_2^r
		\geq 0,\\
		-e^{T}y_i+1=0,\ i\in [K],\\
		0\leq \left(\bar{r}_{2,i}^{T}y_i-\bar{r}_{2,i}^{\min}\right)
		 \perp \pi_{2,i}^r\geq 0,\ i\in [K],\\
		0\leq \left(\tau_{2,i}^{2}-\|x-y_i\|^{2}\right) \perp \pi_{2,i}^\tau\geq 0,\  i\in [K].
	\end{array}
	\right.
\end{eqnarray}

In computation, we will use this concise SVI in ({\ref{two_ncp}}), which will lead to the corresponding concise two-stage SVI for the KKT condition of the original problem (\ref{CP-1}) by Algorithm \ref{algorithm1}. It is easy to see that ({\ref{two_ncp}}) can be written as 
$$0\in H_{\rho_t,x^{t,l},\varpi^{t,l},\tau}^{L}(\mathbf{z}) + N_{D}({\bf z}),$$ where 
$$
\mathbf{z} 
:= (x; y_1; \ldots; y_K; \alpha; \pi_1; \ldots; \pi_K) \in 
D := \mathbb{R}_+^{n} 
\times (\mathbb{R}_+^{n})^{\otimes K} 
\times (\mathbb{R} \times \mathbb{R}_+)
\times (\mathbb{R})^{\otimes K}
\times (\mathbb{R}_+^{2})^{\otimes K},
$$
with $\alpha := (\alpha_1; \alpha_2^r)$ and $\pi_i := (\pi_{1,i}; \pi_{2,i}^r;  \pi_{2,i}^{\tau_{2,i}})$ for $i\in [K]$, and 
\begin{equation}
	\label{H_def}
H_{\rho_t,x^{t,l},\varpi^{t,l},\tau}^{L}(\mathbf{z}) := 
	\begin{pmatrix}
		2Q_{1}x+\frac{1}{\rho_t}x-\varpi^{t,l}+\tau(x-x^{t,l})+\alpha_{1}e-\alpha_{2}^r\bar{r}_{1}+\sum_{i=1}^{K}2\pi_{2,i}^{\tau_{2,i}}(x-y_i) \\
		\frac{2}{K}Q_{2,i}y_i+\pi_{1,i}e-\pi_{2,i}^r\bar{r}_{2,i}+2\pi_{2,i}^{\tau_{2,i}}(y_i-x); \quad i\in[K] \\
		-e^{T}x+1 \\
		\bar{r}_{1}^Tx-\bar{r}_{1}^{\min} \\
		-e^{T}y_i+1;\quad i\in[K] \\
		\bar{r}_{2,i}^{T}y_i-\bar{r}_{2,i}^{\min} ; \quad i\in[K]\\
		\tau_{2,i}^{2}-\|x-y_i\|^{2} ; \quad i\in[K]
	\end{pmatrix}.\nonumber
\end{equation}

The successful implementation of PHM in Algorithm \ref{algorithm1} depends on establishing the result stated in the following proposition.

\begin{proposition}\label{maximal}
	The two-stage SVI {\rm (\ref{two_ncp})} is of maximal monotone type.
\end{proposition}
\begin{proof}
By \cite[Theorem 3.5]{RoWets}, the monotonicity of the two-stage SVI (\ref{two_ncp}) is equivalent to the monotonicity of the decomposed variational inequality for each scenario $i \in [K]$:
	\begin{eqnarray}\label{ncp}
		\left\{\begin{array}{l}
			0\leq\left(2Q_{1}x+\frac{1}{\rho_t}x-\varpi^{t,l}+\tau(x-x^{t,l})+\alpha_{1}e-\alpha_{2}^r\bar{r}_{1}+2\pi_{2,i}^\tau(x-y_i)\right) \perp x \geq 0,\\
			0\leq \left(	\frac{2}{K}Q_{2,i}y_i+\pi_{1,i}e-\pi_{2,i}^r\bar{r}_{2,i}+2\pi_{2,i}^\tau(y_i-x)\right) \perp y_i\geq 0,\\
			-e^{T}x+1 = 0, \\
			0\leq \left(\bar{r}_{1}^Tx-\bar{r}_{1}^{\min}\right) \perp 
			\alpha_2^r \geq 0,\\
			-e^{T}y_i+1=0,\\
			0\leq \left(\bar{r}_{2,i}^{T}y_i-\bar{r}_{2,i}^{\min}\right) \perp \pi_{2,i}^r\geq 0,\\
			0\leq \left(\tau_{2,i}^{2}-\|x-y_i\|^{2}\right) \perp \pi_{2,i}^\tau\geq 0.
		\end{array}
		\right.
	\end{eqnarray}
	  Let  	${\bf{z}}_{\text{p}} := (x; y_i; \alpha_1; \alpha_2^r; \pi_{1,i};
	  {\pi}_{2,i})\in D_{i}:= \mathbb{R}_+^{n} \times  \mathbb{R}_+^{n} 
	  \times \mathbb{R} \times \mathbb{R}_+ \times  \mathbb{R}
	  \times \mathbb{R}_+^2$.  
We only need to show the function ${H}^{i,L}_{\rho_t,{\bf x}_{\bf y}^{t,l},\varpi^{t,l},\tau}({\bf{z}}_{\text{p}})$  associated with the variational 
	  inequality  (\ref{ncp}) is monotone on the set $D_i$, where  
	  \begin{eqnarray*}
	  		\label{H_def}
	  	& &	{H}^{i,L}_{\rho_t,{\bf x}_{\bf y}^{t,l},\varpi^{t,l},\tau}
	  		({\bf{z}}_{\text{p}}) := 
	  		\begin{pmatrix}
	  			2Q_{1}x+\left(\frac{1}{\rho_t}+\tau\right)x-\varpi^{t,l}-\tau x^{t,l}+\alpha_{1}e-\alpha_{2}^r \bar{r}_{1}+2\pi_{2,i}^\tau(x-y_i)\\
	  			\frac{2}{K}Q_{2,i}y_i+\pi_{1,i}e-\pi_{2,i}^r\bar{r}_{2,i}+2\pi_{2,i}^\tau(y_i-x) \\
	  			-e^{T}x+1 \\
	  			\bar{r}_{1}^Tx-\bar{r}_{1}^{\min} \\
	  			-e^{T}y_i+1\\
	  			\bar{r}_{2,i}^{T}y_i-\bar{r}_{2,i}^{\min} \\
	  			\tau_{2,i}^{2}-\|x-y_i\|^{2} \\
	  		\end{pmatrix}.
\end{eqnarray*}

By direct computation, the Jacobian matrix ${\cal{J}}{H}^{i,L}_{\rho_t,{\bf x}_{\bf y}^{t,l},\varpi^{t,l},\tau}
	({\bf{z}}_{\text{p}})$ of	${H}^{i,L}_{\rho_t,{\bf x}_{\bf y}^{t,l},\varpi^{t,l},\tau}
	({\bf{z}}_{\text{p}})$ is given by
	\begin{equation*}	\label{Jacobian_Matrix}
		{\cal{J}}	{H}^{i,L}_{\rho_t,{\bf x}_{\bf y}^{t,l},\varpi^{t,l},\tau}
		({\bf{z}}_{\text{p}}) = 
		\left(\begin{array}{ccccccc}
			2Q_{1}+(2\pi_{2,i}^\tau+\tau+\frac{1}{\rho_t})I_n& -2\pi_{2,i}^\tau I_n & e & -\bar{r}_{1} & 0 & 0 & 2(x - y_i) \\
			-2\pi_{2,i}^\tau I_n & \frac{2}{K}Q_{2,i}+2\pi_{2,i}^\tau I_n & 0 & 0 & e & -\bar{r}_{2,i} & -2(x - y_i) \\
			-e^{T} & 0 & 0 & 0 & 0 & 0 & 0\\
			\bar{r}_{1}^{T} & 0 & 0 & 0 & 0 & 0 & 0\\
			0 & -e^{T} & 0 & 0 & 0 & 0 & 0\\
			0 & \bar{r}_{2,i}^{T} & 0 & 0 & 0 & 0 & 0\\
			-2(x - y_i)^T & 2(x - y_i)^T & 0 & 0 & 0 & 0 & 0
		\end{array}
		\right).
	\end{equation*}
	
Let $\mathbf{u} = (u_x; u_y;u_{\alpha_1}; u_{\alpha_2^r}; u_{\pi_1}; u_{\pi_r}; u_{\pi_\tau}) \in \mathbb{R}^{2n+5}$ be an arbitrary vector partitioned in the same way as the column partition of ${\cal J} {H}^{i,L}_{\rho_t,{\bf x}_{\bf y}^t,\varpi^t,\tau}
({\bf{z}}_{\text{p}})$.
  Then we get  the quadratic form 
	\begin{equation*}
			\mathbf{u}^T 	{\cal{J}}	{H}^{i,L}_{\rho_t,{\bf x}_{\bf y}^{t,l},\varpi^{t,l},\tau}
			({\bf{z}}_{\text{p}}) \mathbf{u} =  u_x^T\left(2Q_{1}+(\tau+\frac{1}{\rho_t})I_n\right)u_x + \frac{2}{K}u_y^T Q_{2,i} u_y  + 2\pi_{2,i}^\tau \|u_x - u_y\|^2\ge 0,
	\end{equation*}
	since $Q_1, Q_{2,i}$ are positive semi-definite matrices, $\rho_t>0, \tau > 0$, and $\pi_{2,i}^\tau \ge 0$.
	Thus, the function ${H}^{i,L}_{\rho_t,{\bf x}_{\bf y}^{t,l},\varpi^{t,l},\tau}({\bf{z}}_{\text{p}})$  associated with the variational 
	inequality  (\ref{ncp}) is monotone on $D_i$ according to  \cite[Proposition 2.3.2]{Finite-Dimensional}.
	The monotonicity of the two-stage SVI (\ref{two_ncp}), combined with the continuity of the involved functions, implies that (\ref{two_ncp}) is of maximal monotone type according to \cite[Example 12.48]{Variational}.
	
\end{proof}

In numerical experiments in the next section, we use the natural residual function, and the variational inequality (\ref{ncp}) in the PHM for each scenario is solved by the efficient semi-smooth Newton method \cite{Luca}.

\section{Numerical experiments}
In this section, we do numerical experiments on the application illustrated in Sect.  \ref{appl} to evaluate the proposed  model and the SDC-PHM algorithm  from multiple perspectives. 
 All experiments are performed in Windows 10 on an Intel Core 10 CPU at 3.70 GHZ with 64 GB of RAM, using MATLAB R2021a. 

To construct the data parameters of the proposed two-stage nonconvex nonsmooth stochastic model (\ref{orimodel1})-(\ref{orimodel2}), we use historical daily returns of $n=40$ stocks selected from the S$\&$P 500 index that  spans from January 5, 2023, to July 1, 2025, comprising a total of $T=500$ daily observations. Let $P_{i,t}$ denote the closing price of the asset $i$ on day $t$. The daily return vector on day $t$ is defined as $R_t := (R_{1,t}; \ldots; R_{n,t}) \in \mathbb{R}^{n}$, where $R_{i,t} = (P_{i,t} - P_{i,t-1}) / P_{i,t-1}$. Using these historical data, the first-stage expected return vector $\bar{r}_1$ and the covariance matrix $Q_1$ are estimated via the following expressions.\begin{equation}
	\begin{aligned}
		& \bar{r}_1 = \frac{1}{T} \sum_{t=1}^{T} R_t, \
		& Q_1 = \frac{1}{T-1} \sum_{t=1}^{T} (R_t - \bar{r}_1)(R_t - \bar{r}_1)^\top + \epsilon I,
	\end{aligned}
\end{equation}where $\epsilon=10^{-9}$ is a regularization parameter introduced to ensure numerical stability.
The second-stage asset returns $\bar{r}_{2,i}$ and the covariance matrices $Q_{2,i}$ for all $i\in [K]$ are generated via a multivariate GARCH
model, which are selected so that 
 the asset returns fall into the $[0.5, 1.5]$ range, volatilities vary around $10\%$, and correlations lie between $-0.5$ and $0.5$, as adopted in \cite{Sanjay}.  
 We consider the equally weighted portfolio $\bar x := \frac{1}{n}{e}$ and set 
 \begin{equation*}
 	\bar{r}_{1}^{\min} := 
 	 \bar{r}_1^T \bar x  - 0.05 |\bar{r}_1^T \bar x|,\quad \bar{r}_{2,i}^{\min}: = 
 	  \bar{r}_{2,i}^T \bar x - 0.05 |\bar{r}_{2,i}^T \bar x|, \ i\in [K].
 	\end{equation*}
 Then assumption A5 holds, since $\bar x$  is a Slater point of the feasible region for (\ref{CP-1}). 

We test different $K\in \{1000,3000,5000\}$. For each choice of $K$, we generate 20 test problems by  independently generating 20 replications of data parameters ${\bar r}_{2,i}$ and $Q_{2,i}$, $i\in [K]$. To ensure reproducibility, distinct random seeds are assigned to each replication. 
Table \ref{tab:problem_sizes} summarizes the dimensions of the test problems used in our experiments. 
\begin{table}[htbp]
	\centering
	\caption{Dimensions of test problems:  number ($\sharp$) of constraints and number of variables. }
	\label{tab:problem_sizes} 
	
	\begin{tabular}{ccc} 
		\toprule 
		& \multicolumn{2}{c}{Problem sizes} \\ 
		\cmidrule(lr){2-3} 
		$K$ & $\sharp$ constraints (42+43$K$) & $\sharp$ variables (40+40$K$) \\ 
		\midrule 
		1,000   &  43,042 &  40,040\\
		3,000   & 129,042  & 120,040 \\
		5,000  & 215,042  & 200,040 \\
		\bottomrule 
	\end{tabular}
\end{table}

In order to analyze the impacts of sparsity-inducing regularization term $\gamma\|x\|_0$ and the SOC constraints, we consider four models labeled A, B, C, and D. Model A is the proposed model (\ref{orimodel1})-(\ref{orimodel2}). Based on Model A, Model B excludes  the SOC constraints; Model C excludes the $l_0$-norm regularization term; and Model D excludes both the SOC constraints and the $l_0$-norm regularization term. Since Model A and Model B both involve the nonconvex and nonsmooth $l_0$-norm, we solve them using the proposed SDC-PHM. As Model C and Model D are convex, we solve them directly using the PHM.

We use the following criteria to measure the performance of computed solution ${\hat{\bf z}}$ together with  $\hat{\lambda} \in \partial \gamma \|\hat x\|_0$, which is constructed based on Lemma \ref{subgradient} and \cite{L0}  as follows:
\begin{equation}
    \hat{\lambda}:=\frac{1}{\hat{\rho}}\hat{x}-\hat{\varpi},
\end{equation}
where $\hat{\rho}$ is the smoothing parameter corresponding to the computed solution $\hat{x}$, and the components of $\hat{\varpi}$ are given by
\begin{equation}
	(\hat{\varpi})_j =
	\begin{cases}
		0, & |(\hat{x})_{j}| < \sqrt{2\gamma\hat{\rho} }, \\
		\frac{(\hat{x})_{j}}{\hat{\rho}}, & |(\hat{x})_{j}| \geq \sqrt{2\gamma\hat{\rho} },
	\end{cases}
	\quad  j \in [n].
\end{equation}

In view of  feasibility and optimality errors defined by the KKT system ($\text{KKT}_{\infty}$ and $\text{KKT}_{\text{rel}}$),  the value relating to SOC constraints (soc), and the feasibility error relating to the primal variables (FeasError), respectively, defined by 
\begin{subequations}
\label{crit_comp}  
\begin{align}
& \text{KKT}_{\infty} := \|H_{D}^{\rm nat}(\hat{\textbf{z}};\hat{\lambda})\|_\infty = \max_i | (H_{D}^{\rm nat}(\hat{\textbf{z}};\hat{\lambda}))_i |, \label{a}\\
& \text{KKT}_{\text{rel}} := \frac{\|H_{D}^{\rm nat}(\hat{\textbf{z}};\hat{\lambda})\|}{1+\|\hat{\textbf{z}}\|}, \label{b}\\
&  \text{soc}(\hat{\bf z}) :=\frac{1}{K} \sum_{i=1}^K  \|x - y_i\|, \label{c}\\
&  \text{FeasErr}(\hat{\bf z}): = \| G_{\rm{eq}}(\hat{\bf z}) \|^2 + \| \max(0, G_{\rm{ineq}}(\hat{\bf z})) \|^2, \label{d}
\end{align}
\end{subequations}
 where $G_{\rm{eq}}$ and $G_{\rm{ineq}}$  are the subvectors of  $G$ corresponding to equality constraints, and inequality constraints, respectively.

In the numerical experiments, we set the regularization parameter $\gamma = 10^{-5}$ in Model A and Model B, $\tau_{2,i} \equiv 0.2$ for all $i\in [K]$ in Model A and Model C, and the parameters for the SDC-PHM in Algorithm \ref{algorithm1} and Algorithm \ref{algorithm2} as follows. $$\tau = 10^{-4} ,\ \eta_1 = \eta_2 =  \eta_3 = K/5,~\sigma = 1 ,$$ and the smoothing parameters
 $$\rho_0=1;\quad \rho_{t+1}=0.8\rho_t\ \mbox{for}\ t\ge 0.$$
To terminate the SDC-PHM method or the PHM method, we first require that
$$|\mbox{obj}_{t-1} - \mbox{obj}_{t}| \le 10^{-3},$$
where ``{obj}'' refers to the objective value of each model that is concerned. 
In addition, for SDC-PHM we also require the smoothing parameter $$\rho_t \le  10^{-4}.$$

We evaluate the computed solution $\hat {\bf z}$  together with  $\hat \lambda$,  obtained from different models in terms of the number of nonzero entries of ${\hat{x}}$ (nnz), the criteria in (\ref{crit_comp}), the
 CPU time, and the total number of  PHM iterations (PHM iter), that refers to the total number of solving the variational inequality (\ref{SVI2}) in Algorithm \ref{algorithm2}, which is the main computation cost in SDC-PHM.
The results reported in Table \ref{ABCD} represent the above performance metrics over these 20 replications.

\begin{table}[h]
\caption{Comparison of computed solutions for models A, B (No SOC), C (No $\ell_0$), and D (No SOC, No $\ell_0$), using mean (without parenthesis) and standard deviation (in parenthesis) over 20 replications.}\label{ABCD}%
\footnotesize 
\setlength{\tabcolsep}{3pt} 
\begin{tabular}{@{}ccccccccc@{}}
\toprule
			$K$ & Model & nnz &$\text{KKT}_{\infty}$ & $\text{KKT}_{\text{rel}} $& $\text{soc}$ & $\text{FeasError}$ & CPU (s) & PHM iter \\
\midrule
\multirow{4}{*}{1,000} 
			& A       & 14 (0) & 7.2e-3 (1.2e-3) & 2.2e-4 (3.8e-5) & 0.17  (1.3e-3)  &  2.1e-5   (7.2e-6)   &  22 (3)    & 55 (6)\\
			& B    & 14 (0)& 3.9e-3 (4.6e-4)&2.9e-2 (4.3e-3)& 0.31 (1.4e-4)   &  2.4e-6  (8.3e-7)    &   16  (1)  & 42 (0)\\
			& C & 32 (2)&4.7e-2 (1.8e-2) &1.7e-2 (9.5e-3)& 0.20 (4.5e-3)   &  1.3e-5  (4.4e-6)    &   117 (18)  & 269 (41)\\
			& D & 24 (1) & 5.0e-2 (2.1e-2)&1.7e-2 (9.2e-3) & 0.28 (2.1e-2)   &  9.0e-6   (4.4e-6)   &  128  (19)   & 301 (43)\\
            			\hline
                        \multirow{4}{*}{3,000} 
			& A       & 14 (0) & 2.2e-2 (6.4e-3)&1.5e-4 (2.0e-5)& 0.17 (1.4e-3)   &  7.8e-5 (2.4e-5)     &36    (5)    & 51 (7) \\
			& B   & 14 (0)& 5.0e-3 (6.8e-4)& 5.0e-2 (7.6e-3)& 0.31 (9.1e-5)   &  6.9e-6  (1.7e-6)    &   30 (2)   & 42 (0)\\
			& C  & 32 (2) & 5.0e-2 (1.7e-2)&9.9e-3 (5.5e-3)& 0.20 (4.5e-3)   &  4.0e-5 (1.3e-5)     &   218 (32)   & 270 (40)\\
			& D & 24 (1) &5.4e-2 (2.0e-2) &9.8e-3 (5.3e-3) & 0.28 (2.1e-2)   &  2.7e-5   (1.3e-5)   &    243 (37)  &301 (43)\\
			\hline
				\multirow{4}{*}{5,000} 
			& A       & 14 (0) &3.7e-2 (1.1e-2) &1.2e-4 (1.4e-5)& 0.17 (1.4e-3)   &  1.4e-4  (4.8e-5)    &   56 (9)   &  50 (8)\\
			& B     & 14 (0)&5.0e-3 (6.8e-4) &6.3e-2 (9.6e-3) & 0.31 (8.4e-5)   &  1.2e-5  (2.9e-6)    &   46 (2)   & 42 (0)\\
			& C  & 32 (1)&5.2e-2 (1.8e-2) &7.7e-3 (4.3e-3) & 0.20 (4.4e-3)   &  6.7e-5 (2.2e-5)     & 357  (56)   & 270 (40)\\
			& D & 24 (1)& 5.6e-2 (2.0e-2)&7.6e-3 (4.1e-3) & 0.28 (2.1e-2)   &  4.6e-5  (2.2e-5)    &    384 (56)  & 302 (43)\\
\botrule
\end{tabular}
\end{table}

From ``nnz'' column in Table \ref{ABCD}, we can find that Models A and B, which incorporate the $\ell_0$-norm regularization, consistently achieve a high level of sparsity with an average of 14 non-zero entries. In contrast, Models C and D, lacking the cardinality constraint, result in much denser portfolios (nnz ranging from 24 to 32). This confirms the effectiveness of the proposed $\ell_0$-regularization term in selecting significant assets.
 From ``soc'' column in  Table \ref{ABCD},  comparing Model A with Model B (and similarly C with D), we observe that the inclusion of SOC constraints significantly reduces the SOC metric (representing the distance between first- and second-stage decisions). For instance, Model A maintains an average SOC value of approximately 0.17, whereas Model B increases to roughly 0.31, exceeding $\tau_{2,i}\equiv 0.2$ for all $i\in [K]$. This result indicates that the SOC constraints effectively ensure that the adjusted decisions in the second-stage do not diverge drastically from the decision in the first-stage.

 From the columns for ``$\text{KKT}_{\infty}$'', ``$\text{KKT}_{\text{rel}}$'', ``FeasError'' in Table \ref{ABCD}, we find that the SDC-PHM successfully generates approximate KKT points for Models A and B, and the PHM successfully generates approximate KKT points for Models C and  D. 
 It is interesting to see that  the proposed SDC-PHM for the nonconvex discontinuous Models A and  B converges much faster than the  PHM  used for the convex Models C and D, as shown in the last two columns of Table \ref{ABCD}. This seems different from what we expected, since the SDC-PHM  uses the PHM as a subroutine, and the $\ell_0$-regularization term is nonconvex and discontinuous that is usually considered to be hard to deal with.
To understand this phenomenon,  we draw in Figure \ref{objvalue} the curves of objective value versus CPU time (blue line), together with the curves of the cardinality of first-stage decision vector versus CPU time (red line), corresponding to the four models for a representative single run with a scenario size of $K=5000$. Figure \ref{objvalue} shows that due to the existence of the $\ell_0$-regularization term, the cardinality decreases much faster in  Models A and B, which leads the objective values of Models A and B to decrease also faster, and the SDC-PHM to converge faster than the PHM for Models C and D.  We also draw in Figure \ref{feasible}  the corresponding curves of feasibility error versus CPU time, which shows that the SDC-PHM for Models A and B achieves acceptable levels of feasibility error as the PHM method for Models C and D.

\begin{figure}[h]  
    \centering     
    \includegraphics[width=1\textwidth]{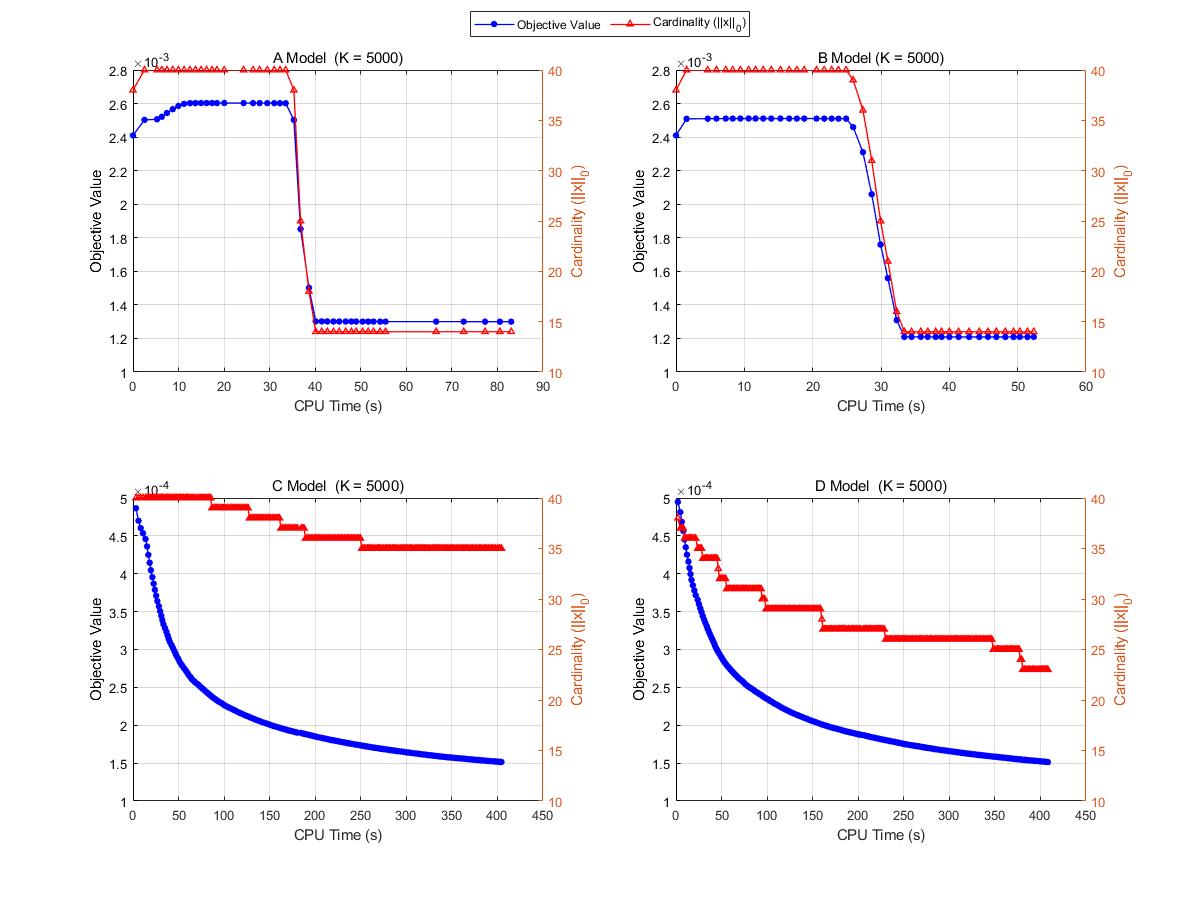} 
    \caption{Objective value (blue) and Cardinality  (red) versus CPU time for Models A–D under $K=5000$} 
  	\label{objvalue} 
\end{figure}

\begin{figure}[h]  
    \centering     
    \includegraphics[width=1\textwidth]{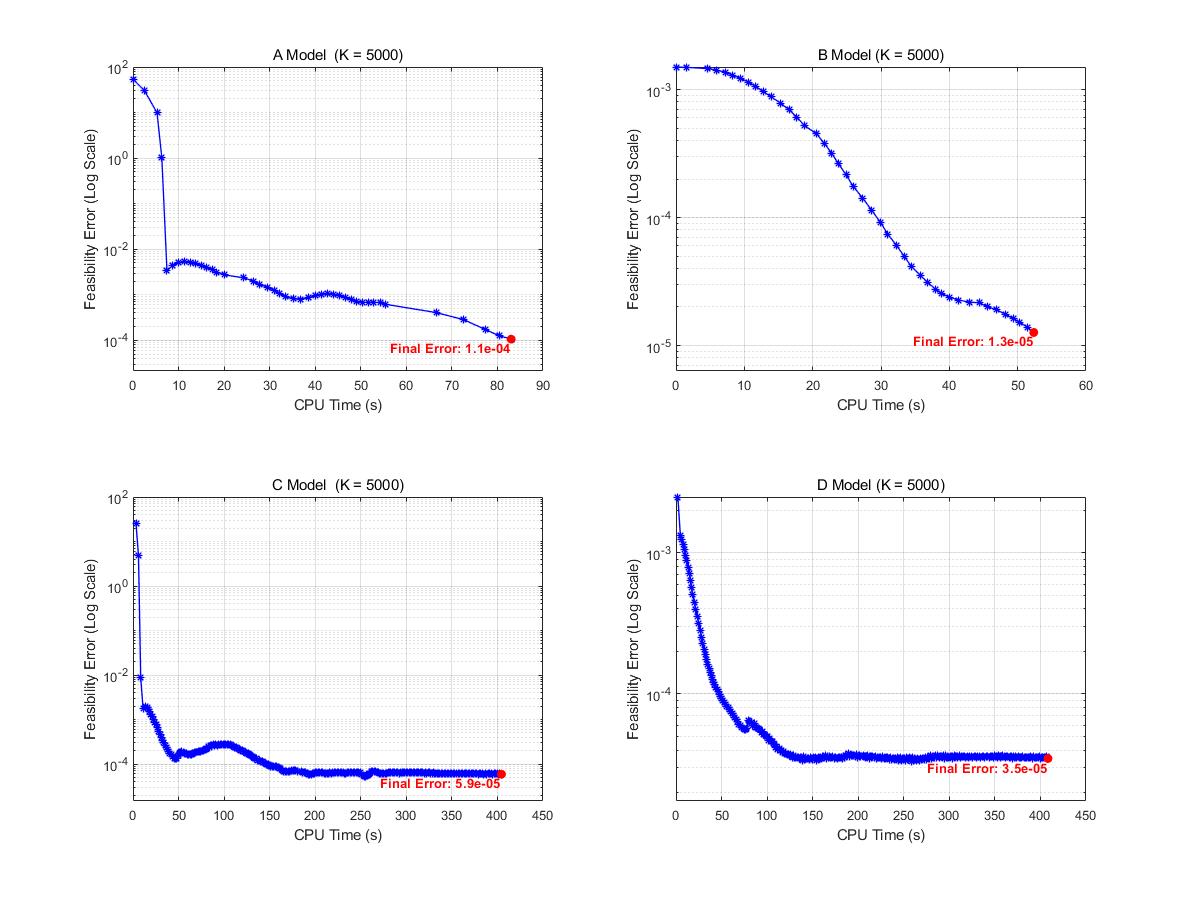} 
    \caption{Feasibility Error versus CPU time for Models A–D under $K=5000$} 
  	\label{feasible} 
\end{figure}

	\newpage	
    
\section{Conclusions}
We  propose the SDC method to solve a class of two-stage nonconvex nonsmooth stochastic conic program via the SVI approach. We show that there exists an accumulation point of the SDC method and any accumulation point is a KKT point, under mild assumptions. The SDC method provides a novel way for dealing with a two-stage nonconvex nonsmooth stochastic conic program, and is flexible to allow for various extensions, to broaden the applicability or enhance the efficiency in computation as illustrated in subsection \ref{extensions}. The SDC method has potential for solving the problem that the second-stage objective function is tangled with the first-stage decision variables, and this will be our future work.	 

\newpage
\backmatter
\bmhead{Acknowledgements} This work was supported by the National Natural Science Foundation of China (Grant No. 12171027; 12571316), and the China Scholarship Council (CSC) (Grant No. 202507090144).

\bibliography{sn-bibliography}


\end{document}